\documentclass[a4paper,10pt,twoside]{amsart}
\usepackage[paper=a4paper,                     % Papierformat
            left=27mm,                         % Seitenrand links
            right=27mm,                        % Seitenrand rechts
            top=30mm,                          % Seitenrand oben
            bottom=30mm,                       % Seitenrand unten
            bindingoffset=0mm                  % Hinzufuegen eines Binderands
           ]{geometry}

\usepackage[english]{babel}
\usepackage{amsmath,amssymb,amscd,amsfonts,amsthm}

\usepackage{caption}
\usepackage{subcaption}

\usepackage{wasysym}

\usepackage{color}
\usepackage[ps,all,arc,rotate]{xy}

\usepackage{mathscinet}
\usepackage{amsrefs}

\usepackage{MnSymbol}

\usepackage{tikz}
\usepackage{tikz-cd}
\usepackage{faktor}

\usepackage{xcolor}
\usepackage{extarrows}

\newtheorem{prop}{Proposition}[section]
\newtheorem{lemma}[prop]{Lemma}

\newtheorem{theorem}[prop]{Theorem}
\newtheorem{cor}[prop]{Corollary}

\newtheorem*{theorem*}{Theorem}

\theoremstyle{definition}
\newtheorem{defn}[prop]{Definition}

\newtheorem{rmk}[prop]{Remark}
\newtheorem{ex}[prop]{Example}

\newtheorem{notation}[prop]{Notation}
\newtheorem*{ex*}{Example}

\newcommand{\bd}{\begin{defn}}
\newcommand{\ed}{\end{defn}}
\newcommand{\bl}{\begin{lemma}}
\newcommand{\el}{\end{lemma}}
\newcommand{\bex}{\begin{ex}}
\newcommand{\eex}{\end{ex}}
\newcommand{\br}{\begin{rmk}}
\newcommand{\er}{\end{rmk}}
\newcommand{\bp}{\begin{prop}}
\newcommand{\ep}{\end{prop}}
\newcommand{\bc}{\begin{cor}}
\newcommand{\ec}{\end{cor}}

\DeclareMathOperator{\Div}{Div}
\DeclareMathOperator{\Pic}{Pic}
\DeclareMathOperator{\Cl}{Cl}
\DeclareMathOperator{\spec}{Spec}
\DeclareMathOperator{\proj}{Proj}
\DeclareMathOperator{\aut}{Aut}
\DeclareMathOperator{\autgr}{Aut_{g}}
\DeclareMathOperator{\Hom}{Hom}

\DeclareMathOperator{\id}{Id}

\DeclareMathOperator{\lie}{Lie }
\DeclareMathOperator{\stab}{Stab }
\DeclareMathOperator{\V}{\textbf{V}}

\DeclareMathOperator{\sset}{\subset}

\DeclareMathOperator{\codim}{codim}

\DeclareMathOperator{\End}{End}
\DeclareMathOperator{\Endgr}{End_{g}}
\DeclareMathOperator{\pr}{pr}
\DeclareMathOperator{\defect}{Def}
\DeclareMathOperator{\diag}{diag}

\DeclareMathOperator{\conv}{Conv}

\DeclareMathOperator{\Span}{Span}
\DeclareMathOperator{\cox}{Cox}

\DeclareMathOperator{\lcm}{lcm}

\DeclareMathOperator{\wt}{wt}

\DeclareMathOperator{\NP}{NP}

\DeclareMathOperator{\D}{\mathbf D}

\DeclareMathOperator{\T}{\mathbf T}

\newcommand{\sslash}{\mathbin{/\mkern-4mu/}}

\newcommand{\longra}{\longrightarrow}
\newcommand{\into}{\hookrightarrow}

\newcommand{\sig}{\sigma}
\newcommand{\Sig}{\Sigma}

\newcommand{\inv}{^{-1}}

\newcommand{\Xs}{X_{\Sigma}}
\newcommand{\oone}{\cO(1)}

\newcommand{\lama}{\lambda_{\underline{a}}}
\newcommand{\YQS}{\cY^{\text{QS}}}
\newcommand{\YNQS}{\cY^{\text{NQS}}}
\newcommand{\xr}{x_{\rho}}
\newcommand{\Xzm}{X_{\min}^0}
\newcommand{\Zm}{Z_{\min}}
\newcommand{\QS}{\text{QS}}

\def\cL{\mathcal L}
\def\cM{\mathcal M}\def\cO{\mathcal O}
\def\cT{\mathcal T}

\def\cY{\mathcal Y}

 \def\s{\text{s}} \def\ss{\text{ss}}  \def\SM{\text{SM}}\def\fingen{\text{fg}}

\def\AA{\mathbb A}
\def\GG{\mathbb G}

\def\PP{\mathbb P}
\def\QQ{\mathbb Q}

\def\ZZ{\mathbb Z}

\def\fg{\mathfrak g}
\def\fl{\mathfrak l}

\def\fu{\mathfrak u}

\def\fC{\mathfrak C}

 \def\GL{\mathrm{GL}} \def\SL{\mathrm{SL}} \def\PGL{\mathrm{PGL}} 

\usetikzlibrary{shapes.misc}
\tikzset{cross/.style={cross out, draw=black, fill=none, minimum size=2*(#1-\pgflinewidth), inner sep=0pt, outer sep=0pt}, cross/.default={2pt}}

\makeatletter

\newcommand{\xdashrightarrow}[2][]{\ext@arrow 0359\rightarrowfill@@{#1}{#2}}
\newcommand{\xdashleftarrow}[2][]{\ext@arrow 3095\leftarrowfill@@{#1}{#2}}
\newcommand{\xdashleftrightarrow}[2][]{\ext@arrow 3359\leftrightarrowfill@@{#1}{#2}}
\def\rightarrowfill@@{\arrowfill@@\relax\relbar\rightarrow}
\def\leftarrowfill@@{\arrowfill@@\leftarrow\relbar\relax}
\def\leftrightarrowfill@@{\arrowfill@@\leftarrow\relbar\rightarrow}
\def\arrowfill@@#1#2#3#4{%
  $\m@th\thickmuskip0mu\medmuskip\thickmuskip\thinmuskip\thickmuskip
   \relax#4#1
   \xleaders\hbox{$#4#2$}\hfill
   #3$%
}
\makeatother

\title{On the moduli of hypersurfaces in toric orbifolds}
\author{Dominic Bunnett}
\date{}
\begin{document}
\begin{abstract}
We construct and study the moduli of stable hypersurfaces in toric orbifolds. Let $X$ be a projective toric orbifold and $\alpha \in \Cl(X)$ an ample class. The moduli space is constructed as a quotient of the linear system $|\alpha|$ by $G = \aut(X)$.  Since the group $G$ is non-reductive in general, we use new techniques of non-reductive geometric invariant theory. Using the $A$-discriminant of Gelfand, Kapranov and Zelevinsky, we prove semistability for quasismooth hypersurfaces of toric orbifolds. Further, we prove the existence of a quasi-projective moduli space of quasismooth hypersurfaces in a weighted projective space when the weighted projective space satisfies a certain condition. We also discuss how to proceed when this condition is not satisfied. We prove that the automorphism group of a quasismooth hypersurface of weighted projective space is finite excluding some low degrees.
\end{abstract}
\maketitle

\section{Introduction}

We study the moduli of hypersurfaces in toric orbifolds.
The main tool we use to do this is geometric invariant theory (GIT), both reductive and non-reductive.
We work over an algebraically closed field $k$ of characteristic 0.

Suppose that $X$ is a projective variety and $\alpha \in \Cl(X)$ is an ample class.
We are interested in the quotient stack
\[\cM(X,\alpha) = \left[ \, \faktor{|\alpha|}{G} \, \right] \]
which we define as the moduli stack of hypersurfaces of class $\alpha$ in $X$ where $G=\aut(X,\alpha)$, the automorphisms of $X$ fixing $\alpha$.
As it is, this stack has very little chance of possessing any desirable properties and, as with all moduli problems, we must define a notion of stability which will instruct us to remove certain hypersurfaces to obtain a stack with a more manageable geometry.
In this paper, we apply a non-reductive GIT-stability condition to the case where $X$ is a toric orbifold, which provides a coarse moduli space (which is a variety) for a substack and prove in a number of cases that quasismooth hypersurfaces are semistable.
Further we shall be able to prove stability when $X$ falls into a class of certain weighted projective spaces.
From this point on, when we refer to a moduli space of hypersurfaces, we mean a coarse moduli space of a substack of $\cM(X,\alpha)$.

The moduli space of stable hypersurfaces in projective space was constructed by Mumford using reductive GIT.
Consider the $n$-dimensional projective space $\PP^n$ and let $d>2$ be an integer.
The linear system $Y_d = |dH| = \PP(k[x_0 , \dots , x_n]_d)$ is a parameter space for all hypersurfaces of degree $d$ in $\PP^n$ and the action of the reductive group $\aut(\PP^n) = \PGL_{n+1}$ on $\PP^n$ extends naturally to an action of $\PGL_{n+1}$ on $Y_d$.
It turns out to be hard to describe the open set of stable points in $Y_d$, let alone the actual quotient variety.
However, what Mumford does prove is that if $d>2$ (and $d>3$ if $n=1$), then a smooth hypersurface is stable.
Thus reductive GIT constructs a moduli space of smooth hypersurfaces.

By a toric orbifold we mean a projective toric variety with at worst orbifold singularities.
One cannot mirror the construction of moduli spaces of hypersurfaces given above for toric orbifolds: the algebraic groups in question are in general non-reductive. Recent work of Kirwan, B\'erczi, Doran and Hawes \cites{monster, BDHK_gr_uni,DBHK_proj_compl} develops a non-reductive GIT and allows one to construct such moduli spaces as non-reductive quotients. 

There has been much work in search of GIT for non-reductive groups; see \cite{DK} for a comprehensive account of work undertaken in this area. We use the non-reductive GIT (NRGIT) developed in \cites{BDHK_gr_uni,DBHK_proj_compl,monster} where one requires that the unipotent radical of the group is graded (Definition \ref{grd_uni_rad_def}). For these theorems to apply, we must assume that a property we call the $(\fC)$ condition holds; see Definition \ref{ss_equals_s}.

In \cite{cox_hom}, Cox showed that the automorphism group of a complete simplicial toric variety can be calculated from graded automorphisms of the Cox ring. We use Cox's construction and show that the automorphism group of a toric orbifold admits a graded unipotent radical (see Proposition \ref{pos_grad} and \cite[Section 4]{BDHK_gr_uni}) and thus the theory of NRGIT is applicable to the problem of moduli of hypersurfaces in toric orbifolds.

Let $X$ be a weighted projective space where the condition $(\fC)$ is satisfied for the action of $\aut(X)$. Theorem \ref{thm_main_NP_proof} proves that a Cartier quasismooth hypersurface in $X$ is stable. Let $\cY_d$ be the parameter space of degree $d$ hypersurfaces. The group $\aut(X)$ acts on $\cY_d$ and we denote the stable locus by $\cY_d^{\s}$ and the quasismooth locus by $\cY_d^{\QS}$. We use NRGIT to construct a quotient space of such hypersurfaces which is a coarse moduli space. In particular, this coarse moduli space is a quasi-projective variety.

\begin{theorem*}[Theorem \ref{thm_main_NP_proof}]
Let $X = \PP(a_0, \dots , a_n) = \proj k[x_0, \dots , x_n]$ be a well-formed weighted projective space and let $d \geq \max\{a_0, \dots , a_n\} + 2$. Suppose that the $(\fC)$ condition holds for the action of $G=\aut(X)$ on $\cY_d = \PP(k[x_0 , \dots , x_n]_d)$. Then a quasismooth hypersurface of degree $d$ is a stable hypersurface. In other words, there is an inclusion of open subsets
\[\YQS_d \sset \cY_d^{\s}.\]
In particular, there exists a geometric quotient $\YQS_d / G$ and hence a coarse moduli space of quasismooth hypersurfaces of degree $d$ in $X$. Moreover, the NRGIT quotient $\cY_d \sslash \! G$ is a compactification of $\YQS_d / G$.
\end{theorem*}

The proof of the theorem relies on a discrete-geometric version of the Hilbert-Mumford criterion for NRGIT. We also provide an explicit construction for quasismooth hypersurfaces in $X = \PP(1, \dots , 1,r)$ which works without the heavy machinery of NRGIT.

To study the quasismooth locus in a given linear system, one uses the $A$-discriminant, defined and studied in \cite{GKZ} for general toric varieties and denoted by $\Delta_A$. We show that $\YQS_d \sset (\cY_d)_{\Delta_A}$ and discuss the possible difference between these two sets. We prove in Section \ref{sec_A_discrim} that the $A$-discriminant can be interpreted as an invariant section of an appropriate line bundle, just as for the classical discriminant.

\begin{theorem*}[Corollary \ref{invar_A_discrim}]
Let $X$ be the toric variety associated to a polytope $P$ and let $A$ be the lattice points of $P$. The $A$-discriminant $\Delta_{A}$ is a semi-invariant section for the $G$-action on $\cY_\alpha$ and a true $U$-invariant, where $U \sset G$ is the unipotent radical of $G$.
\end{theorem*}

Restricting our attention to a weighted projective space $X$, we prove in Section 3 that the stabiliser groups of the action of $G=\aut(X)$ on $\YQS_d$ is finite for $d \geq \max(a_0, \dots , a_n)+2$. Denote the stabiliser group by $\aut(Y;X)$.

\begin{theorem*}[Theorem \ref{finite_stabilisers_thm}]
A quasismooth hypersurface in $X = \PP(a_0, \dots , a_n)$ of degree \linebreak $d \geq \max\{a_0, \dots , a_n\} + 2$ has only finitely many automorphisms coming from the automorphisms of the ambient weighted projective space. That is, the group $\aut(Y;X)$ is finite for a quasismooth hypersurface $Y \sset \PP(a_0,\dots,a_n)$.
\end{theorem*}

A corollary of this theorem is the existence of a moduli space as an algebraic space. This is a direct consequence of the Keel-Mori theorem. However, Theorem \ref{thm_main_NP_proof} implies that this algebraic space is in fact a quasi-projective variety.

There are many different classes of varieties which present themselves as hypersurfaces in weighted projective spaces; for example, genus 2 curves are degree 6 curves in $\PP(1,1,3)$, Petri special curves are degree 6 curves in $\PP(1,1,2)$ and degree 2 del Pezzo surfaces are degree 4 surfaces in $\PP(1,1,1,2)$ to name a few. Hence we find constructions of new moduli spaces or new constructions of well-known moduli spaces.

There are many applications of the constructions we provide in this paper. The moduli spaces and notions of stability we present in this paper are related to K-stability and mirror symmetry and can be seen as a natural extension of work of Batyrev and Cox \cite{cox_hodge}.

\noindent\textbf{Acknowledgements} This work forms part of the author's thesis which was funded by the FU Berlin.
The author is very grateful to my supervisor Victoria Hoskins for her support and guidance over the
last few years.
The final touches to this paper were made whilst at Oxford university and the author would also like to thank Frances Kirwan for the hospitality.
Finally, the author would like to thank Joshua Jackson for many useful discussions and bringing to my attention a possible generalisation of Proposition 5.9.

%%%%%%%%%%%%%%%%%%%%%%%%%%%%%%%%%%%%%%%%%%%%%%%%%%%%%%
% NRGIT
%%%%%%%%%%%%%%%%%%%%%%%%%%%%%%%%%%%%%%%%%%%%%%%%%%%%%%

\section{Non-reductive geometric invariant theory}

In this section we introduce non-reductive GIT and present some recent results due to B\'erczi, Doran, Hawes and Kirwan \cites{DBHK_proj_compl,BDHK_gr_uni}. We call these results the $\hat{U}$-Theorems and they will be the main tools used in constructing the moduli spaces of hypersurfaces in complete simplicial toric varieties.

The method adopted in \cites{BDHK_gr_uni, DBHK_proj_compl} requires additional structure on the algebraic groups and the linearisations chosen. With this additional structure, many of  the properties of reductive GIT can be recovered. In this section we now introduce and explore this additional structure and state the resulting theorems. It must be noted that we introduce definitions and state theorems in only as much generality as is required. The definitions and results hold in greater generality than is stated and we refer the reader to \cites{DBHK_proj_compl,BDHK_gr_uni} for statements in full generality.

Let $X$ be a projective variety acted on by a linear algebraic group $G$ with respect to a very ample linearisation $\cL$, where $G$ is not necessarily reductive. Let $G \simeq R \ltimes U $ be the Levi decomposition.

\begin{defn}
We define the morphism of schemes associated to the inclusion of graded rings 
\[A(X,\cL)^G \sset A(X,\cL)\]
to be the \em enveloping quotient \em 
\[ q_G : X -\!\rightarrow  X \sslash_{\!\cL}  G,\]
where $X \sslash_{\!\cL}  G = \proj A(X,\cL)^G$ is a scheme, not necessarily of finite type.
\end{defn}

We define notions of semistability and stability for linear algebraic group actions. One motivation of this definition is to have a quotient locally of finite type. Due to the presence of a global stabiliser for the actions we will be considering, we need our definition of stability to allow for a positive dimensional stabiliser, so we adopt a variant of the definition given in \cites{DBHK_proj_compl,BDHK_gr_uni}, in analogy to the construction of the moduli space of quiver representations \cite{king}. We do this as it is easier to work a modified definition of stability rather than with than the group resulting from quotienting out by this global stabiliser.

\bd\label{def_nrgit_stab}
Let $X$ be a projective variety acted on by a linear algebraic group $G$ with respect to a very ample linearisation $\cL$. Suppose that $D \sset G$ is a torus acting trivially on $X$.
We define
\[I^{\fingen} = \{ \sig \in A(X,\cL)^G_+ \, \, | \, \, \cO(X_\sig) \text{ is finitely generated}\}\]
and the \em finitely generated semistable locus \em to be
\[X^{\ss} = \bigcup_{\sig \in I^{\fingen}} X_\sig.\]
Further, we define $I^s \sset I^{\fingen}$ to be $G$-invariant sections satisfying the following conditions:
\begin{itemize}
\item the action of $G$ on $X_f$ is closed and for every $x \in X_f$ we have  $D \sset \stab_G(x)$ with finite index; and
\item the restriction of the $U$-enveloping quotient map
\[q_U : X_\sig \longrightarrow \spec(\cO(X)_{(\sig)}^U)\]
is a principal $U$-bundle for the action of $U$ on $X_\sig$.
\end{itemize}
Then we define 
\[X^{\s} = \bigcup_{\sig \in I^{\s}}X_\sig\]
to be the stable locus.
\ed

\begin{notation}
When there is a possibility of confusion, we write $X^{\s,G}$ and $X^{\ss,G}$ for $X^{\s}$ and $X^{\ss}$ respectively when we want to emphasise the group.
\end{notation}

\br
When the torus $D \sset G$ is trivial, Definition \ref{def_nrgit_stab} coincides with the original definition of \cite{BDHK_gr_uni}. Alternatively, if the group $G$ is reductive, so that $U = \{e\}$, the definition agrees with \cite{king}.
\er

The following lemma details how we may study the quotient of an action of $G$ on $X$ in two stages. If we first deal with the action of $U$ on $X$, then we may consider the action of $R$ on $X / U$, provided it exists, using reductive GIT.

\bl\cite[Lemma 3.3.1]{monster}\label{lemma_quotients_in_stages}
Suppose $G$ is a linear algebraic group, N is a normal subgroup of $G$ and $X$ is a scheme with a $G$-action. Suppose all the stabilisers for the restricted action of $N$ on $X$ are finite and this action has a geometric quotient $\pi : X \to X/N$. Note that $G/N$ acts canonically on $X/N$. Then the following statements hold.

\begin{enumerate}
\item For all the $G/N$-orbits in $X/N$ to be closed, it is necessary and sufficient that all the $G$-orbits in $X$ are closed;
\item given $y \in X/N$, the stabiliser $\stab_{G/N}(y)$ is finite if and only if $\stab_G(x)$ is finite for some (and hence all) $x \in \pi\inv(y)$; and
\item if $G/N$ is reductive and $X/N$ is affine, then $X/N$ has a geometric $G/N$-quotient if and only if all $G$-orbits in $X$ are closed.
\end{enumerate}
\el

Using Lemma \ref{lemma_quotients_in_stages}, we can construct a geometric quotient of the stable locus as defined in Definition \ref{def_nrgit_stab}.

\begin{theorem}\cite[Theorem 3.4.2]{monster}\label{thm_monster_nrgit}
Let $X$ be a projective variety and $G$ a linear algebraic group acting on $X$ with respect to a very ample line bundle. There is a commutative diagram

\[
     \begin{tikzcd}
		X^{\s} \arrow[r] \arrow[d] & X^{\ss} \arrow[r] \arrow[d] & X \arrow[ld]\\
		X^{\s}/G \arrow[r] & X^{\ss} \sslash G
      \end{tikzcd}
    \]

where the first arrow is a geometric quotient and all inclusions are open.
\end{theorem}
The question remains of how one can compute the stable and semistable locus. The following discussion aims to address this.

Let $U$ be a unipotent group and $\lambda : \GG_m \to \aut(U)$ be a 1-parameter subgroup of automorphisms and let 
\[\hat{U}_{\lambda} = \GG_m \ltimes_{\lambda} U\]
be the semi-direct product, where multiplication is given as follows:
\[(u_1,t_1) \cdot (u_2,t_2) = (\lambda({t_2\inv})(u_1) + u_2, t_1t_2), \quad u_i \in U, \, \, t_i \in \GG_m.\]
The pointwise derivation of $\lambda$ defines a $\GG_m$-action on $\lie U$. This action defines a grading $\lie U = \bigoplus_{i \in \ZZ}(\lie U)_{i}$ with respect to weights $i \in \ZZ = \Hom(\GG_m,\GG_m)$.

\bd\label{grd_uni_rad_def}
We say that $\hat{U}_{\lambda}$ is \em positively graded \em if the induced action of $\GG_m$ on $\lie U$ has all positive weights. That is $(\lie U)_{i} \neq 0$ implies that $i>0$. 

Let $G \simeq R \ltimes U$ be a linear algebraic group. We say that $G$ has a \em graded unipotent radical \em if there exists a central 1-parameter subgroup $\lambda_g : \GG_m  \to R$ such that $\eta : \GG_m \to \aut(U)$ defined by 
\[\eta (u) =  \lambda_g(t) \cdot u \cdot \lambda_g(t)\inv \quad \quad \text{ for } t \in \GG_m, \,\, u \in U,\]
is such that $\hat{U}_{\lambda_g}:=\hat{U}_{\eta}$ is positively graded. We often drop the grading 1-parameter subgroup from the subscript and write $\hat{U}_{\lambda_g} = \hat{U}$. Note that $\lambda_g(t)$ is an automorphism of $U$ since $U$ is a normal subgroup of $G$.
\ed

Let $X$ be a projective variety and $\cL \in \Pic(X)$ be a very ample line bundle. Suppose that $\hat{U}$ acts on $X$ with respect to $\cL$. By restricting the $\hat{U}$-action to $\GG_m$, we have a $\GG_m$-action on $V = H^0(X,\cL)^\vee$; let
\[\omega_{\min} = \text{ minimial weight in } \ZZ \text{ for the } \GG_m \text{-action on }V\]
and
\[V_{\min} = \{v \in V \, \, | \, \, t \cdot v = t^{\omega_{\min}}v \text{ for all } t \in \GG_m\}\]
the associated weight space. Then $\PP(V_{\min})$ is a linear subspace of $\PP(V)$.

\bd\label{Zmin_Xmin0_def}
Suppose that $X, \cL$ and $\hat{U}$ are as above. We define 
\[Z_{\min} = X \cap \PP(V_{\min}) \]
and 
\[X_{\min}^0 = \{ x \in X \, \, | \, \, \lim_{t \to 0} t \cdot x \in Z_{\min}\} \quad \text{ where } t \in \GG_m \sset \hat{U}.\]
\ed

\br 
The subvarieties $\Zm$ and $\Xzm$ are unaffected by replacing the linearisation $\cL$ by any element of the positive $\QQ$-ray defined by $\cL$ in $\Pic^{\hat{U}}( X )\otimes_\ZZ \QQ$. Also note that $\Xzm$ and the $U$-sweep $U \cdot \Zm$ of $\Zm$ are $\hat{U}$-invariant subsets.
\er

We may twist the linearisation $\cL$ by a character $\chi : \hat{U} \to \GG_m$. We denote the twisted linearisation by $\cL^\chi$ and the minimum $\GG_m$-weight of $V = H^0(X,\cL^\chi)^\vee = H^0(X,\cL)^\vee$ by $\omega^\chi_{\min}$.

\bd\label{def_ep_chara}
We say that a linearisation $\cL \in \Pic^{\hat{U}}(X)$ is \em adapted \em if we have the following inequality 
\[\omega_{\min} < 0 < \omega_{\min +1},\]
where $\omega_{\min}$ is the lowest $\GG_m$-weight and $\omega_{\min +1}$ is the second smallest weight.
\ed
\br\label{rmk_find_well_adapt_lin}
Fix $\cL \in \Pic^{\hat{U}}(X)$ a linearisation and note that every character of $\hat{U} = \GG_m \ltimes U$ is of the form
\[\hat{U} \to \GG_m \, ; \quad (t,u)\mapsto t^r,\]
for some $r\in \ZZ$. We identify the characters of $\hat{U}$ with $\ZZ$. Let $\epsilon >0$ be a rational number and consider the rational character $\chi = -\omega_{\min} - \frac{\epsilon}{2}$. Twist $\cL$ by the character $\chi$ and denote this linearisation $\cL^{\chi} \in\Pic^{\hat{U}}(X)_{\QQ}$. Then for $\epsilon > 0$ small enough, $\cL^{\chi}$ is adapted: indeed, we have that
\[\omega_{\min}^{\chi} = \omega_{\min} + \chi = -\frac{\epsilon}{2} < 0 < \omega_{\min}^{\chi}+\epsilon.\]
\er

\br\label{rmk_defn_well_adpated}
For the proofs of finite generation of invariants given in \cites{BDHK_gr_uni,DBHK_proj_compl} to work, we must twist the linearisation by a rational character so that it is not merely adapted but so that the weight $\omega_{\min}$ is within some sufficiently small $\epsilon>0$ of the origin. That is, $\omega_{\min} < 0 < \omega_{\min} + \epsilon$. We say that a linearisation is \em well-adapted \em when this inequality holds for a sufficiently small $\epsilon >0$. By Remark \ref{rmk_find_well_adapt_lin}, so long as the linearisation is adapted, it is always possible to twist the linearisation to make it well-adapted. We refer the reader to \cites{BDHK_gr_uni,variationNRGIT} for a more in depth discussion.
\er

Before we state the $\hat{U}$-theorem, there is a technical condition which we require.
\bd \label{ss_equals_s}
The $G$-action is said to satisfy the \em semistability equals stability \em condition if 
\begin{align}\tag{$\fC$}
\stab_U(z) = \{e\} \text{ for every } z \in \Zm^{\ss,R} \enspace ,
\end{align}
where $\Zm^{\ss,R}$ is the semistable locus for the induced linear action of $R$ on $\Zm$.
\ed

Note that the condition $(\fC)$ as it appears here is the generalised condition \cite[Remark 2.7]{DBHK_proj_compl}.

\begin{theorem}\cite[Theorem 2]{BDHK_gr_uni}\label{thm_Uhat_thm}
Let $X$ be a projective variety acted on by a graded unipotent group $\hat{U}$ with respect to a very ample linearisation $\cL$. Suppose that the action satisfies the condition $(\fC)$. Then the following statements hold.
\begin{enumerate}
\item The restriction to $\Xzm$ of the enveloping quotient for the $U$-action 
\[q_U : \Xzm \longrightarrow \Xzm / \, U\]
is a principal $U$-bundle, in particular, $q_{U}$ is a geometric quotient.
\end{enumerate}
Suppose furthermore that $\Xzm \neq U \cdot \Zm$ and that the linearisation $\cL$ is well-adapted, then the following statements hold.

\begin{enumerate}

\setcounter{enumi}{1}

\item There are equalities $\Xzm -  U \cdot \Zm = X^{\s,\hat{U}}(\cL)$.
\item The enveloping quotient $X \sslash_{\cL}  \hat{U}$ is a projective variety and 
\[q_{\hat{U}} : X^{\s , \hat{U}}(\cL) \longrightarrow X \sslash_{\cL} \hat{U}\]
is a geometric quotient for the $\hat{U}$-action. In particular, the ring of $\hat{U}$-invariants is finitely generated.
\end{enumerate}
\end{theorem}

\br\label{rmk_triv_U_quot}
It follows from the proof of the $\hat{U}$-hat theorem that if $Z_{\min}$ is a point (so that $\dim V_{\min} = 1$), then $\Xzm = X_\sig$ for some non-zero section $\sig \in (V_{\min})^\vee$. Thus $\Xzm$ is an affine open subscheme of $X$. Moreover, when this is the case, the quotient $q_U : \Xzm \to \Xzm / U$ is a trivial $U$-bundle.
\er

We now state the result for general linear algebraic groups. Let $G \cong R \ltimes U$ be a linear algebraic group with unipotent radical $U \sset G$. Suppose that there exists a 1-parameter subgroup $\lambda_g:\GG_m \to R$ lying in the center of the Levi factor of G such that $\hat{U} = \lambda_g(\GG_m)\ltimes U$ is a graded unipotent group.
\bd\label{def_full_group_ep_chara}
A linearisation of the $G$-action is said to be \em (well-)adapted \em if its restriction to $\hat{U}$ is a (well-)adapted in the sense of Definition \ref{def_ep_chara}.
\ed

\begin{theorem}\cite[Theorem 0.1]{DBHK_proj_compl}\label{non_red_git_thm}
Let $G$ be a linear algebraic group acting on a projective variety $X$ with respect to $\cL$. Assume that $G$ has graded unipotent radical such that $(\fC)$ holds. Further, assume that $\cL$ is well-adapted. Then the following statements hold.
\begin{enumerate}
\item The $G$-invariants are finitely generated and the enveloping quotient 
\[X \sslash_{\cL} G = \proj A(X,\cL)^G\]
is a projective variety.

\item The inclusion $A(X,\cL)^G \sset A(X,\cL)$ induces a categorical quotient of the semistable locus
\[X^{\ss,G} \longrightarrow X \sslash_{\cL} G,\]
which restricts to a geometric quotient
\[X^{\s,G} \longrightarrow X^{\s,G} / G.\]
\end{enumerate}
\end{theorem}

We now state a Hilbert-Mumford criteron, whose proof is outlined in \cite{DBHK_proj_compl}.

\begin{theorem}\label{non_red_HM_crit}\cite[Theorem 2.6]{DBHK_proj_compl}
Keep the notation and assumptions as in Theorem \ref{non_red_git_thm}. The following Hilbert-Mumford criterion holds.
\[X^{(\s)\s , \, G} = \bigcap_{g \in G} gX^{(\s)\s, \, T},\]
where $T\sset G$ is a maximal torus of $G$ containing the grading $\GG_m$.
\end{theorem}

\subsection{Weight polytopes}\label{sec_s_global_stab}
Using Theorem \ref{non_red_HM_crit}, we can reduce the study of stability to the study of stability of a maximal torus.  The discrete-geometric version of the Hilbert-Mumford criterion for a torus described in \cite{mumford} must be adapted to work in the presence of global stabilisers. Suppose that a torus $T$ is acting on a projective space $X=\PP^n $ with respect to a very ample linearisation $\cO(1)$. Consider the canonical identification $X= \PP(V)$, where $V = H^0(X , \cO(1))^\vee$.  Suppose further that there exists a 1-parameter subgroup
\[\lama:\GG_m \longrightarrow T\]
such that $\lama(\GG_m) \sset \stab_T(x)$ for every $x \in X$. Consider the weight space decomposition
\[V = \bigoplus_{\chi \in X^*(T)} V_{\chi},\]
where $X^*(T) = \Hom(T,k^*)$ is the character group and $V_\chi = \{ v \in V \, \, | \, \, t\cdot v = \chi(t)v \,\, \forall t \in T \}$.  The 1-parameter subgroup $\lama$ defines a point in $W = X^*(T)\otimes_\ZZ \QQ$, denote this point by $\underline{a} \in W$. Define the quotient vector space $H_{\underline{a}} = W / \QQ \cdot \underline{a}$ and write $\overline{w} \in H_{\underline{a}}$ for the image of an element $w \in W$ in $H_{\underline{a}}$. 

Let $\overline{T} = T / \lama(\GG_m)$ and consider $x \in X$ and some $v \in V$ lying over $x$ and write $v = \sum v_\chi$. Note that we can equivalently construct $H_{\underline{a}} = X^*(\overline{T}) \otimes_\ZZ \QQ$. We define the $\overline{T}$-weight set of $x$ to be
\[\wt_{\overline{T}}(x) = \{\overline{\chi} \, \, | \, \, v_\chi \neq 0\} \sset H_{\underline{a}},\]
and the associated weight polytope to be the convex hull of these weights:
\[\conv_{\overline{T}}(x) = \conv(\chi \, \, | \, \, \chi \in \wt_T(x)) \sset H_{\underline{a}}.\]
We get the discrete-geometric Hilbert-Mumford Criterion for (semi)stability with respect to the torus in the presence of a global stabilising $\GG_m$.

\begin{theorem}[Reductive Hilbert-Mumford criterion]\label{thm_stab_HM_polytope}
Let $T$ be a torus acting on a projective scheme $X$ with linearisation $\cL$ such that there is a global stabiliser $\lama : \GG_m \to T$ acting trivially on $X$. Then
\begin{align*}
x \in X^{\ss,T}(\cL)  &\iff 0 \in \conv_{\overline{T}}(x) , \\
x \in X^{\s,T}(\cL) &\iff 0 \in \conv_{\overline{T}}( x)^\circ,
\end{align*}
where $\conv_{\overline{T}}(x)^\circ$ is the interior of the polytope.
\end{theorem}

Combining this with Theorem \ref{non_red_HM_crit}, we have a non-reductive Hilbert-Mumford criterion.

\begin{theorem}[Non-reductive Hilbert-Mumford criterion]\label{thm_nrgit_hm_2}
Let $G$ be a linear algebraic group acting on a projective variety $X$ with respect to $\cL$. Assume that $G$ has graded unipotent radical such that $(\fC)$ holds. The following Hilbert-Mumford criterion holds.
\begin{align*}
x \in X^{\ss,G} & \iff 0 \in \conv_T(g \cdot x) \text{ for every } g \in G, \\
x \in X^{\s,G} & \iff 0 \in \conv_T(g \cdot x)^\circ \text{ for every } g \in G. \\
\end{align*}
\end{theorem}

\bex\label{ex_wt_polytope_preshifting_wpp112}
We consider degree 4 curves in $\PP(1,1,2)$, such curves are parametrised by the projective space of weighted forms $\cY = \PP(k[x,y,z]_4)$, where $\deg x = \deg y = 1$ and $\deg z =2$. Then $\GL_2$ acts on $(x,y)$ via matrix multiplication and $\GG_m$ acts on $z$ via multiplication. This defines an action of $G =\GL_2 \times \GG_m$ on $\cY$. Consider the maximal torus of $G$ defined by
\[T = \left\{ \left( \diag(t_1,t_2), s \right) \in \GL_2 \times \GG_m \, \, | \, \, t_1,t_2,s \in k^* \right\}.\]
Consider the restricted action of $T$ on $\cY$. Then for a general monomial $x^iy^jz^k \in k[x,y,z]_4 = V$ with $i+j+2k = 4$, we have that 
\[(t_1,t_2,s) \cdot x^i y^j z^k = t_1^i t_2^j s^k x^i y^j z^k.\]
Denote such a weight by $(i,j,k)\in X^*(T) \cong \ZZ^3$. Note that by collecting all possible weights as columns in a matrix, one gets the following matrix
\[A = \begin{pmatrix}
4&3&2&1&0&2&1&0&0\\
0&1&2&3&4&0&1&2&0\\
0&0&0&0&0&1&1&1&2
\end{pmatrix}, \]
which is the matrix from Example \ref{ex_112_A}.

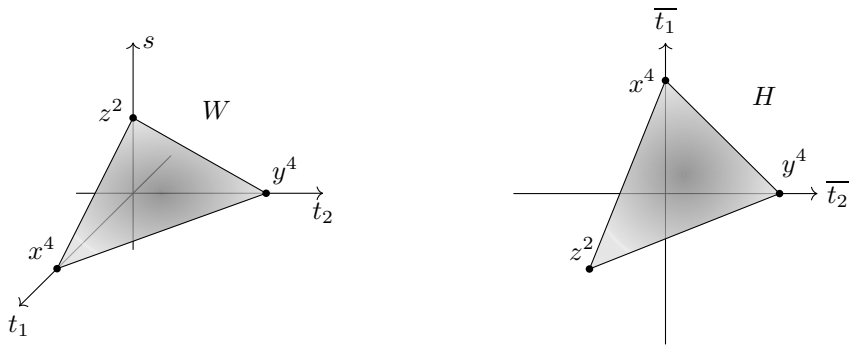
\begin{figure}[h]
\centering
\begin{tikzpicture}

\draw  [->] (0.5,0.5)-- (-1.5,-1.5);
\draw  [->] (0,-0.75) -- (0,2);
\draw [->] (-0.75,0) -- (2.5,0);

\fill[shading=radial, nearly transparent] (-1,-1) to (1.75,0)  to (0,1);

\node [below] at (-1.5,-1.5) {$t_1$};
\node [above] at (-1.2,-1) {$x^4$};
\node [circle, fill, inner sep=1.0pt] at (-1,-1) {};
\node [right] at (0,2) {$s$};
\node [left] at (0,1.1) {$z^2$};
\node [circle, fill, inner sep=1.0pt] at (0,1) {};
\node [below] at (2.5,0) {$t_2$};
\node [above] at (2,0) {$y^4$};
\node [circle, fill, inner sep=1.0pt] at (1.75,0) {};

\draw  [-] (-1,-1)-- (1.75,0);
\draw  [-]  (1.75,0) -- (0,1);
\draw [-] (0,1) -- (-1,-1);
\node at (1.1,1.1) {$W$};

\end{tikzpicture}
\hspace{2cm}
\begin{tikzpicture}

\draw  [->] (0,-2) -- (0,2);
\draw [->] (-2,0) -- (2,0);
\fill[shading=radial, nearly transparent] (-1,-1) to (1.5,0)  to (0,1.5);

\node [above] at (-1.1,-1) {$z^2$};
\node [circle, fill, inner sep=1.0pt] at (-1,-1) {};

\node [left] at (0,1.5) {$x^4$};
\node [circle, fill, inner sep=1.0pt] at (0,1.5) {};

\node [above] at (1.7,0.1) {$y^4$};
\node [circle, fill, inner sep=1.0pt] at (1.5,0) {};

\draw  [-] (-1,-1)-- (1.5,0);
\draw  [-]  (1.5,0) -- (0,1.5);
\draw [-] (0,1.5) -- (-1,-1);

\node [above] at (0,2) {$\overline{t_1}$};
\node [right] at (2,0) {$\overline{t_2}$};

\node at (1.3,1.3) {$H$};

\end{tikzpicture}
\caption{$P_4$, the section polytope of $\cO(4)$ in $W$ and $H$.}\label{fig_sec_poly_W_and_H}
\end{figure}

Define $P_4$ to be the \em section polytope \em to be the convex hull of all torus weights in either $W$ or $H = H_{\underline{a}}$ (see Figure \ref{fig_sec_poly_W_and_H}). The weight polytope of any element $[f] \in \cY_4$ will be a subpolytope of the section polytope. As shown in figure one, the section polytope considered in $W$ is not full-dimensional, and is a 2-simplex in $W \cong \QQ^3$. We will always consider the section polytope in $H$, where it is full dimensional.

Explicitly, if $e_1$ and $e_2$ are a basis for $H$, then $P_4 = \conv(4e_1,4e_2,-e_1-e_2)$.
\eex

%%%%%%%%%%%%%%%%%%%%%%%%%%%%%%%%%%%%%%%%%%%%%%%%%%%%%%
% AUTS AND TORIC VARIETIES
%%%%%%%%%%%%%%%%%%%%%%%%%%%%%%%%%%%%%%%%%%%%%%%%%%%%%%

\section{Automorphisms and toric varieties}
In this section, we study the automorphism groups of toric varieties and prove that they carry the extra structure required by NRGIT; that is, that they admit graded unipotent radicals. We first recall some aspects of toric geometry; specifically toric orbifolds and their quasismooth hypersurfaces.

\subsection{Toric orbifolds}
Let $X$ be a projective simplicial toric variety. These are precisely the projective toric varieties with at worst orbifold singularities and we refer to such varieties simplt as toric orbifolds.  The notion of quasismoothness was first introduced for subvariaties of weighted projective spaces by Dolgachev \cite{dolgachev_weighted} and then generalised to toric varieties by Cox and Batyrev \cite{cox_hodge}.

We recall the Cox ring and the quotient construction of a toric variety. We refer the reader to \cites{cox_hom,hausen_book} for details. 

\begin{defn}
Suppose that $X = \Xs$ is a toric variety associated to a fan $\Sig$. The 1-dimensional cones in $\Sig$ are called rays and the set of rays is denoted $\Sig(1)$. Let
\[S = k[x_{\rho} \, |  \, \rho \in \Sig(1)]\]
be the polynomial ring in $|\Sig(1)|$ variables. Every monomial $\prod x_{\rho}^{a_{\rho}} \in S$ defines an effective torus-invariant divisor $D = \sum a_{\rho} D_{\rho}$, we write this monomial as $x^D$. In this way, we define the following notion of degree:
\[\deg(x^D) = [D] \in \Cl(X).\]
Thus we have 
\[S = \bigoplus_{\alpha \in \Cl(X)}S_{\alpha},\]
where $S_\alpha = \{f \in S \, \, | \, \, \text{ all monomials of } f \text{ of degree } \alpha \}$. Then $S_{\alpha} \cdot S_{\beta} \subset S_{\alpha + \beta}$ and we define the \em Cox ring of $X$ \em to be $S$ with this grading.
\end{defn}

Now when $X$ is a projective orbifold, every Weil divisor is $\QQ$-Cartier and for every $\alpha \in \Cl(X)$ we denote the corresponding rank 1 reflexive sheaf by $\cO_X(\alpha)$. By  \cite[Proposition 1.1]{cox_hom}, the Cox ring of $X$ is the algebra of global sections of the rank 1 reflexive sheaves and denoted by 
\[S = \bigoplus_{\alpha \in \Cl(X)} H^0(X,\cO_X(\alpha)).\]

Since $X$ is toric, $\Cl(X)$ is a finitely generated abelian group. The grading of $S$ by $\Cl(X)$ defines an action of $\D = \Hom_\ZZ(\Cl(X),k^*)$ on $\AA^r = \spec S$ (for example, see \cite{craw}). Cox proved (\cite[Theorem 2.1]{cox_hom}, although for the way we have formulated the action the proof can be found in \cite[Theorem 2.12]{craw} that $X = \left( \AA^r \right)^{\s} / \D$ as the GIT-quotient of this action. We denote the the quotient morphism $q : (\AA^r)^{\s} \to X$.

\begin{defn}\label{def_qs}
Let $X$ be a toric orbifold and fix a class $\alpha \in \Cl(X)$. Let $Y \sset X$ be a hypersurface defined by $f \in S_{\alpha}$. We say that $Y$ is \em quasismooth \em if $q\inv(Y) \sset (\AA^{r})^{\s}$ is smooth. Equivalently, $Y$ is quasismooth if $\V(f) \sset \AA^r$ is smooth in $(\AA^r)^{\s}$.

Let $\cY_\alpha = \PP(S_\alpha)$. We define the \em quasismooth locus \em to be the open set
\[  \YQS_\alpha = \{Y \sset X\, \, | \, \, Y \text{ is a quasismooth hypersurface of class }\alpha \} \sset \cY_\alpha.\]
We denote its complement, \em the non-quasismooth locus, \em by
\[\YNQS_\alpha = \cY_\alpha - \YQS.\]
\end{defn}

\bex
Let $X = \PP(a_0, \dots , a_n)$ be a weighted projective space. Then $S = \bigoplus_{d\geq 0}k[x_0,\dots , x_n]_d$ where $\deg x_i = a_i$ and $q : \left(\AA^{n+1} - \{0\} \right) \to \PP(a_0,\dots ,a_n)$ is the quotient morphism for the $\GG_m$-action on $\AA^{n+1}$ defined by $t \cdot (x_0, \dots , x_n) = (t^{a_0}x_0,\dots , t^{a_n}x_n)$. Moreover, a hypersurface defined by $f \in k[x_0, \dots , x_n]_d$ is quasismooth if and only if $\left( \frac{\partial f}{\partial x_0}(\tilde{x}) , \dots , \frac{\partial f}{\partial x_n}(\tilde{x}) \right) \neq (0,\dots ,0)$ for every $\tilde{x} \in \AA^{n+1} - \{0\}$.
\eex

\begin{theorem}\cite[Theorem 8.1]{fletcher}\label{thm_flecther_qs_criterion}
The general hypersurface of degree $d$ in $\PP(a_0, \dots a_n)$ is quasismooth if and only if 
\begin{itemize}
\item[] either (1) there exists a variable $x_i$ of degree $d$,
\item[] or (2) for every non-empty subset $I = \{i_0, \dots , i_{k-1}\}$ of $\{0, \dots , n\}$, \\  either (a) there exists a monomial $x_I^M = x_{i_0}^{m_0} \cdots x_{i_{k-1}}^{m_{k-1}}$ of degree $d$ \\  or (b) for $\mu = 1, \dots , k$ there exists monomials
\[x_I^{M_\mu}x_{e_\mu} = x_{i_0}^{m_{0,\mu}} \cdots x_{i_{k-1}}^{m_{k-1,\mu}}x_{e_\mu}\]
of degree $d$, where $\{e_\mu\}$ are $k$ distinct elements.
\end{itemize}
\end{theorem}

The following lemma will be needed for the proof of Theorem \ref{thm_main_NP_proof}. It is a more expicit version of Theorem \ref{thm_flecther_qs_criterion} and allows us to remove the general hypothesis. It follows from arguments used in the proof of Theorem \ref{thm_flecther_qs_criterion} in \cite{fletcher}.

\bl\label{lemma_fletcher_qs_crit}
Suppose that $X = \PP(a_0, \dots , a_n)$ be well formed and $d$ a Cartier degree and denote $d_i = \frac{d}{a_i}$. Consider an $f \in k[x_0, \dots , x_n]_d$ which is quasismooth. Then for every variable $x_i$, either $x_i^{d_i}$ is a monomial of $f$ or $x_i^{d_i - \frac{a_j}{a_i}}x_j$ is a monomial of $f$ for some $j \neq i$ where $a_i | a_j$.
\el

\br
Note that if for a fixed $a_i$, there exists no $a_j$ such that $a_i | a_j$, then we must have that $x_i^{d_i}$ is a monomial of $f$.
\er

\subsection{Graded automorphisms of the Cox ring}

The construction of the automorphism group of a complete simplicial toric variety $X$ is a generalisation of the construction of the automorphism group of projective space. The generalisation is to be seen as follows. The Cox ring of projective space with the grading of the class group is the standard homogeneous coordinate ring; that is, the polynomial ring with the usual $\ZZ$-grading given by the total degree. The group of graded automorphisms of this ring is $\GL(n+1)$ which fits into the following short exact sequence
\[0 \longra \GG_m \longra \GL(n+1) \longra \PGL(n+1) \longra 0,\]
where $ \PGL(n+1) = \aut(\PP^n)$. More generally, let $X$ be a complete toric variety associated to a fan $\Sig$. Let $S = k[\xr \, | \, \rho \in \Sig(1)]$ be the Cox ring of $X$. When we refer to the degree of an element of $S$, we mean the degree with respect to the class group and by total degree we mean the degree with respect to the usual $\ZZ$-grading of the polynomial ring. 
We obtain a short exact sequence
\[0 \longra \D \longra \autgr(S) \longra \aut^0(X) \longra 0.\]

\begin{theorem}\cite[Theorem 4.2]{cox__hom_erratum}\label{auto_group_toric_cox}
Let $X$ be a complete toric variety and let $S = \cox (X)$ be its Cox ring. Then the following statements hold.

\begin{enumerate}

\item The group of graded algebra automorphisms $\autgr(S)$ is a connected affine algebraic group of dimension $\sum_{i=1}^l |\Sig_i| \dim_k S_i$.
\item The unipotent radical $U$ of $\autgr(S)$ is of dimension $\sum_{i=1}^l |\Sig_i| (\dim_k S_i - |\Sig_i|)$.
\item We have the following isomorphism
\[ \autgr(S) \cong \prod_{i=1}^l \GL(S_i') \ltimes U.\]

\end{enumerate}
\end{theorem}

We refer the reader to \cite{cox__hom_erratum} for a proof.  We consider the group of graded automorphisms as a matrix group via the following lemma, whose proof is taken from the proof of the corrected version of \cite[Proposition 4.3]{cox__hom_erratum}. We include the proof since we will need the explicit matrix description of the automorphism group.

%The labeling of the lemma comes from the orginal paper of Cox. He eventually names 6 or 7 properties, we shall need more.

\bl\label{E2}

The endomorphism algebra of $S$ is a linear algebraic monoid with unit group $\autgr(S)$ and there is an inclusion of linear algebraic monoids
\begin{align*}
\Endgr(S) & \longrightarrow  \prod_{i=1}^l \End_k(S_i)\\
\phi & \longmapsto (\phi|_{S_i} : S_i \to S_i)_{i=1}^l.
\end{align*}
In particular, $\autgr(S)$ is a linear algebraic group.
\el

\begin{proof}
We show that the map
\[\Endgr(S) \longrightarrow \prod_{i=1}^l \End_k(S_i)\]
is a closed immersion and hence $\End_g(S)$ is an affine submonoid. Since $S$ is generated as an algebra by elements in $S_1,...,S_l$, an endomorphism is completely determined by the above restrictions and hence the map is injective. The fact that the map respects composition (and is well-defined) is immediate since we consider only graded endomorphisms. Thus $\Endgr(S)$ is a submonoid and it only remains to show that it is a closed subset; that is, cut out by polynomials.

To do this, we write down the corresponding collection of matrices with respect to the basis of each $S_i = S_i' \oplus S_i''$ given by monomials of degree $\alpha_i$:
\[
\phi \longleftrightarrow
\left( \left( \begin{array}{c|c}

A_i & 0 \\
\hline
B_i & C_i\\

\end{array} \right)\right)_{i=1}^l, \label{eq:E2_2} \tag{$\star$}
\]
where $B_i \in \Hom_k(S_i' , S_i'')$.  We shall often suppress the brackets in this notation.

The matrices $A_i$ and $B_i$ come from evaluating the single variables in $S_i'$. The $C_i$ come from evaluating monomials in $S_i''$ which are products of 2 or more variables in $S_j'$ with $j\neq i$ and hence $C_i$ is completely determined by $A_j$ and $B_j$ for $j\neq i$. We claim that the elements of $C_i$ are polynomials in elements of $A_j$ and $B_j$ for $j \neq i$.

Let us prove this claim.  Consider monomials $x^D,x^E \in S_i''$, where $D$ and $E$ are effective non-prime divisors with class $\alpha_i$. Both $x^D$ and $x^E$ are elements of the monomial basis of $S_i''$ so that for $\phi \in \Endgr(S)$
\[\phi(x^D) = \cdots + c_i^{DE}x^E + \cdots,\]
where $c_i^{DE}$ is the corresponding entry in $C_i$. Then $x^D = x_{\rho_1} \cdots x_{\rho_s}$ is a product of variables allowing duplications with $x_{\rho_i} \nin S_i'$. Thus
\[\phi(x_{\rho_1}) \, \cdots \,\phi(x_{\rho_s}) = \cdots +  c_i^{DE}x^E + \cdots.\]
But each $\phi(x_{\rho_k})$ is a linear combination of monomials with coefficients given by elements of $A_{j}$ and $B_{j}$ with $j \neq i$. Thus the elements of the $C_i$ are given by polynomials in the elements of $A_{j},B_{j}$ and we are done.

On the other hand, the $A_i$ and $B_i$ are chosen completely arbitrarily. In other words we have a bijection of sets
\begin{align*}
\Endgr(S) & \longleftrightarrow \prod_{i=1}^l \Hom_k(S_i',S_i) \\
\phi & \longleftrightarrow  \begin{pmatrix}
A_i\\
B_i
\end{pmatrix}.
\end{align*}
The $0$ in the top right hand corner of the matrices \eqref{eq:E2_2} comes from the fact that  
\[\phi(S_i'') \cap S_i' = 0\]
since $S_0 = k$, and monomials in $S_i''$ contain more than one variable. 

It remains to remark that $\autgr(S)$ is the group of invertible elements in a linear algebraic monoid. It follows from \cite[Corollary 3.26]{monoids} that $\autgr(S)$ is a linear algebraic group.

\end{proof}

\begin{prop}\label{pos_grad}

The unipotent radical $U$ of $\autgr(S)$ is given by matrices of the form 
\[\left(\begin{array}{c|c}
I_i & 0 \\
\hline
B_i & C_i
\end{array} \right) \]
under the correspondence in \eqref{eq:E2_2}, where $C_i$ are lower triangular matrices with 1's on the diagonal. 

Moreover, the 1-parameter subgroup given by
\begin{align*}
\lambda_g : \, \,  \GG_m & \longrightarrow  \autgr(S) \\
  t & \longmapsto (\phi_t : \xr \mapsto t\inv \xr)
\end{align*}
gives $U$ a positive grading. We refer to $\lambda_g$ as the distinguished $\GG_m$.
\end{prop}
\br 
Note that this result was already given in the paper \cite{BDHK_gr_uni}. The proof uses the original incorrect construction of the automorphism group given in the paper \cite{cox_hom}. We present a proof using the corrected construction given in \cite{cox__hom_erratum}.
\er
\begin{proof}
It is clear that the matrices above form a unipotent subgroup and we refer the reader to \cite[Theorem 4.2]{cox__hom_erratum} for a proof that it is in fact the unipotent radical. We prove that it is positively graded by the distinguished $\GG_m$.

Under \eqref{eq:E2_2} we have 
\[\lambda_g(t) \longleftrightarrow \left( \begin{array}{c|c}
t\inv I_i & 0 \\
\hline
0 & Q_i(t)
\end{array}\right) \]
where 
\[Q_i(t) = \diag(t^{-l_1},...,t^{-l_k})\]
are diagonal matrices with $l_j \geq 2$. To see this, consider $x^D = x_{\rho_1} \cdots x_{\rho_l} \in S_i''$ again allowing duplications. Then
\[\lambda_g(t)(x^D) = \lambda_g(t)(x_{\rho_1}) \cdots \lambda_g(t)( x_{\rho_l}) = t^{-l}x^D\]
where $l$ has to be greater than 2 since $D$ was a non-prime divisor.

To calculate the weights on the Lie algebra of $U$ consider the conjugation action
\[
\lambda_g(t\inv) \left(\begin{array}{cc}
I_i & 0 \\
B_i & C_i
\end{array} \right)
\lambda_g(t) = 
\left(\begin{array}{c|c}
I_i & 0 \\
\hline
t  Q_i(t\inv)B_i & Q_i(t)C_iQ_i(t)
\end{array} \right)
\]
of an arbitrary element of $U$ by $\lambda_g(t)$. Then the matrix in the bottom left hand corner is given by
\[\begin{pmatrix}
t^{l_1 -1} && \\
& \ddots & \\
& & t^{l_k-1}
\end{pmatrix}B_i\]
and since each $l_j \geq 2$, the exponents here are strictly positive. This suffices to show that the group is graded unipotent since the matrices $B_i$ describe the Lie algebra.
\end{proof}

\br\label{struc_auts_wps}
For a general weighted projective space $X = \PP(a_0, \dots , a_n) = \proj S$, where $S = k[x_0, \dots , x_n]$, we describe the Levi factor of the group $G = \autgr(S)$ explicitly. First, partition the variables $x_i$ into distinct weights $\Sig_j = \{ x_i \, \, | \, \, \deg x_i = a_j\}$ and set $n_i = |\Sig_i|$. Then the Levi factor of $G$ is equal to
\[\prod_{\Sig_i} \GL_{n_i} \sset G,\]
where the product is taken over the distinct $\Sig_i$. Thus the Levi factor contains all linear automorphisms: that is, automorphisms which take variables to linear combinations of other variables. As an automorphism must respect the grading, these linear combinations only contain variables of the same weight.

The unipotent radical of $G$ is given by `non-linear' automorphisms: that is, automorphisms which involve a monomial of total degree higher than 1.
\er

Let $X = \PP(a_0, \dots , a_n) = \proj k[x_1, \dots , x_{n'} , y_1, \dots , y_{n_l}]$ be a weighted projective space and let $G = \autgr(S)$ be as above. Assume that the weights are in ascending order (so that $a_i \leq a_{i+1}$) and label the distinct weights $b_1 < \cdots < b_l$ where each $b_j$ occurs exactly $n_i$ times (the $n_i$ coincide with the $n_i$ in Remark \ref{struc_auts_wps}). For weighted projective space we define another 1-parameter subgroup which grades the unipotent radical $U \sset G$ positively depending on a parameter $N\in \ZZ$.

\bp\label{prop_alternative_grading_wps}
Let $N>0$ be a positive integer. The 1-parameter subgroup $\lambda_{g,N} : \GG_m \to G$ defined by
\[\lambda_{g,N} : t \longmapsto \left( \left(t^{-N}I_{n_i}\right)_{i=1}^{l-1} , tI_{n_l} , 0\right)\]
gives $U \sset G$ a positive grading.
\ep

\begin{proof}
Let $X = \PP(a_0 ,\dots a_n) = \proj S$ where $S = k[x_0, \dots , x_{n'} , y_0,\dots , y_{n_l}]$ so that the $y_i$ have the maximum weight $b_l = a_n$. Then $\lambda_{g,N}(\GG_m) \sset G = \autgr(S)$ acts on $X$ as follows:
\begin{align*}
\lambda_{g,N}(t) \cdot (0: \dots :x_i: \dots : 0)  &=   (0: \dots :t^{-N}x_i:0: \dots : 0)\\
\lambda_{g,N}(t) \cdot (0: \dots :y_j: \dots : 0) &=  (0: \dots: 0 :ty_j: \dots : 0)  
\end{align*}
for $0\leq i \leq n'$ and $0 \leq j \leq n_l$. Let $u \in U \sset G$ be an element of the unipotent radical. By Remark \ref{struc_auts_wps}, $u$ acts on $S$ as follows
\begin{align*}
u \cdot x_i &= x_i + p_i(x_0, \dots , x_{n'})\\
u \cdot y_j &= y_j + q_j(x_0, \dots , x_{n'}),
\end{align*}
for $0\leq i \leq n'$ and $0 \leq j \leq n_l$, where $p_i,q_j \in k[x_0,\dots ,x_{n'}]$ are weighted homogeneous polynomials (possibly 0) of degree $a_i$ and $a_n$ respectively. Note that $p_i = 0$ for those $i$ such that $a_i = b_1$ is the minimum weight and that $p_i$ and $q_j$ do not contain any factors of $y_j$, since the $y_j$ all have the same maximal weight. In particular, if $p_i \neq 0$, then $\deg p_i >1$.

Consider the action by conjugation of $\lambda_{g,N}(\GG_m)$ on $U$, first on the $x_i$:
\begin{align*}
\left(\lambda_{g,N}(t) \cdot u \cdot \lambda_{g,N}(t\inv) \right)\cdot x_i &= \left(\lambda_{g,N}(t) \cdot u \right)\cdot t^{N}x_i\\
& = \lambda_{g,N}(t) \cdot (t^{N}x_i + p_i(t^{N}x_0, \dots , t^{N}x_{n'}))\\
& = x_i + t^{-N} p_i(t^{N}x_0, \dots , t^{N}x_{n'}).
\end{align*}
Those $p_i$'s which are non-zero have degree $a_i>1$ and hence if $u$ is a weight vector for the $\lambda_{g,N}(\GG_m)$-action, it has weights $a_iN - N>0$. The argument for the $y_i$ is identical and is omitted.
\end{proof}

\subsection{Finiteness of the stabilisers}

Let $S =\cox(X)= k[x_0,...,x_n]$ be the Cox ring of the weighted projective space $X=\PP(a_0,...,a_n)$ and assume that $a_0 \leq a_1 \leq \cdots \leq a_n$. Label the distinct values of the $a_i$'s by $b_1, ..., b_l$ such that $b_1 < \cdots < b_l$. Define numbers $n_1,...,n_l$ such that each of the $b_j$ occur exactly $n_j$ times, so that the $n_j$ sum to $n+1$. Recall from Theorem \ref{auto_group_toric_cox} that
\[\autgr(S) = \prod_{j=1}^l \GL_{n_j} \ltimes U,\]
where $U$ is the unipotent radical.
Let $G=\autgr(S)$ and denote the 1-parameter subgroup of $G$ by
\[\lama : t \longmapsto ((t^{b_j} I_{n_j})_{j=1}^l, 0).\]
Since the automorphism group of weighted projective space is connected, we have that
\[\aut(\PP(a_0,...,a_n)) \cong \autgr(S) / \lama(\GG_m).\]

\begin{theorem}\label{finite_stabilisers_thm}
Let $S = k[x_0,...,x_n]$ be the polynomial ring with the weighted grading $\deg x_i = a_i$ and let $f \in k[x_0,...,x_n]_d$ define a quasismooth hypersurface $\V(f) \sset \PP(a_0,\dots,a_n)$ where $d \geq \max\{a_i\}+2$. Define the subgroup $\aut(f) \sset \autgr(S)$ as follows
\[\aut(f)  = \{ \phi \in \autgr(S) \, | \, \V(\phi(f)) = \V(f)\}.\]
Then $\aut(f) =\mu \ltimes \lama(\GG_m)$, with $\lama : \GG_m \to G$ defined as above and $\mu$ is a finite group.
\end{theorem}

\begin{proof}

We write $G  = \autgr(S) = \prod_{j=1}^l \GL_{n_j} \ltimes U$ and denote $\lie G$ by
\[\fg = \prod_{j=1}^l \fg\fl_{n_j} \ltimes \fu.\]

It is clear that $\lama(\GG_m) \sset \aut(f)$. To obtain the desired result it suffices to show that the Lie algebras of $\aut(f)$ and $\lama(\GG_m)$ agree as sub-Lie algebras of $\fg$.

The Lie algebra $\fg$ acts on $S$ by derivation: let $\xi \in \fg$ and $F \in S$ be arbitrary elements of $\fg$ and $S$ respectively, then
\[\xi(F) = \sum_{i = 0}^n F_i \, \xi(x_i),\]
where $F_i = \frac{\partial F}{\partial x_i}$.
Suppose that $\xi \in \lie(\aut(f)) \sset \fg$. Then since $f$ is semi-invariant under the action of $\aut(f)$, it is also a semi-invariant for the action of $\lie(\aut(f))$; that is, $\xi(f) = \tilde{\alpha} f$ for some $\tilde{\alpha} \in k$. The weighted Euler formula tells us that $f = \frac{1}{d}\sum_{i=0}^n a_if_i$ and so
\[\sum_{i=0}^n f_i (\xi (x_i) - \alpha a_i x_i) = 0,\]
where $\alpha = \frac{\tilde{\alpha}}{d}$.

Rearranging, for each $i$ we get an equation
\[p_if_i = -(p_0f_0 + \cdots +p_{i-1}f_{i-1}+p_{i+1}f_{i+1} + \cdots + p_nf_n),\]
where $p_j = \xi(x_j) - \alpha a_j x_j$. Thus $p_i f_i \in(f_0, ... ,f_{i-1},f_{i+1}, ... , f_n)$.

Since $f$ is quasismooth, its partial derivatives $f_0,...,f_n$ form a regular sequence. Moreover, any permutation of the $f_i$ is a regular sequence. Thus $f_i$ is a non-zero divisor in the ring $S / (f_0, ... ,f_{i-1},f_{i+1}, ... , f_n)$ and hence $p_i \in (f_0, ... ,f_{i-1},f_{i+1}, ... , f_n)$. However, $\deg p_i = a_j$ and since we assumed $\deg f \geq \max\{a_j\}+2$, this forces $p_i = 0$ and $\xi(x_i) = \alpha a_i x_i$. Thus $\alpha$ is the only parameter and we have shown that $\lie(\aut(f))$ is one dimensional and hence agrees with that of $\lie \lama(\GG_m)$.

Moreover, we can see explicitly that
\[\lie(\aut(f)) = \{((\alpha b_j I_{n_j})_{j=1}^l,0) \, | \, \alpha \in k\} \sset \fg,\]
which is precisely the Lie algebra of $\lama(\GG_m)$.
\end{proof}

\br 
Let $X = \PP(a_0, \dots , a_n)$ be a weighted projective space $d \geq \max(a_0, \dots , a_n)+2$ be an integer. 
Then the algebraic stack $\cM(X,d) = \left[ \YQS_d / \aut(X) \right]$ admits a coarse moduli space as an algebraic space.
This is an immediate consequence of Theorem \ref{finite_stabilisers_thm} and the Keel-Mori Theorem \cite[Corollary 1.2]{keelmori1997}.
\er

%%%%%%%%%%%%%%%%%%%%%%%%%%%%%%%%%%%%%%%%%%%%%%%%%%%%%%
% A DISCRIMINANT
%%%%%%%%%%%%%%%%%%%%%%%%%%%%%%%%%%%%%%%%%%%%%%%%%%%%%%

\section{The $A$-discriminant of a toric variety}\label{sec_A_discrim}

\subsection{The $A$-discriminant}

Consider a torus $(k^*)^{r+1}$ with coordinates $(x_0,\dots,x_r)$ and consider a matrix
\[A = ( \omega^{(0)} \, | \,\cdots \, | \, \omega^{(N)} ) \in \ZZ^{(r+1) \times (N+1)},\]
where $\omega^{(j)} \in \ZZ_{\geq 0}^{r+1}$ is a column vector for $0 \leq j \leq N$. Define the vector space of Laurent functions on $(k^*)^{r+1}$ associated to $A$ by
\[k^A := \Big\{\sum_{i=0}^N a_i x^{\omega^{(i)}} \, \, | \, \, a_i \in k \Big\} .\]
Here $x^{\omega^{(i)}} = x_0^{\omega^{(i)}_0}\cdots \, \, x_n^{\omega^{(i)}_r}$, where $\omega^{(i)}$ is a column vector defined to be the transpose of $(\omega^{(i)}_0,\dots,\omega^{(i)}_r)$.

\begin{defn}\label{A_discrim_def}
Consider the following subset of $\PP(k^A)$ consisting of Laurent functions (up to scalar multiple) with a singular point on the torus
\[ \nabla_A^\circ = \Big\{ f \in \PP(k^A)  \, \, | \, \, \exists \, x \in (k^*)^{r+1} \text{ s.t. } f(x) = \frac{\partial f}{\partial x_i}(x) = 0 \text{ for all } i=0,\dots , n \Big\}.\]
Then define the \em $A$-discriminant locus \em to be 
\[\nabla_A= \overline{\nabla_A^\circ} \sset \PP(k^A).\]
As before, we define 
\[\defect A = \codim_{\PP(k^A)} (\nabla_A) - 1.\] 
If $\defect A = 0$, then define the \em $A$-discriminant \em $\Delta_A$ as the polynomial defining $\nabla_A$ which is well defined and unique up to a scalar multiple. If the codimension is greater than 1, we set $\Delta_A = 1$. We shall only work with embeddings where $\defect A = 0$ and we shall always assume that this is the case. We refer to \cite[Corollary 1.2]{GKZ} for a geometric characterisation of this property.
\end{defn}

We now apply this theory in the context of toric varieties.

\bd\label{toric_A_disrim}
Let $X = \Xs$ be a toric variety and $\alpha \in \Cl(X)$ and $S = k[x_0,\dots,x_r]$ be the Cox ring of $X$. Let $N = \dim S_\alpha - 1$.  We define a matrix $A_{\Sig,\alpha} \in \ZZ^{(r+1) \times (N+1)}$ by collecting the exponents of the monomial basis of $S_\alpha$ as columns of this matrix with respect to some ordering of the monomials. We define the \em $A$-discriminant associated to $X$ and $\alpha$ \em to be $\Delta_{A_{\Sig,\alpha}}$. When it is clear from context, we shall drop the $\Sig$ and $\alpha$ from the subscript and write simply $A = A_{\Sig,\alpha}$.
\ed

\br
Let $X$ and $\alpha$ be as above, then $k^A = S_\alpha$ and hence
\[\nabla_A \sset \PP(S_\alpha).\]
\er
Let $X$ be a simplicial projective toric variety and suppose that $\alpha$ is a very ample class. The corresponding $A$-discriminant is a special case of the discriminant as defined in \cite[Chapter 1]{GKZ} in terms of the projective dual of a variety, as the following proposition shows.

\begin{prop}\label{A_discrim_locus}
Let $X$ be a projective toric variety, $\alpha \in \Cl(X)$ be a very ample class and \linebreak$A=A_{\Sig,\alpha} \in \ZZ^{(r+1) \times (N+1)}$ be the associated matrix of exponents of the monomial basis of $S_\alpha$. Then 
\[\nabla_A = X^{\vee, \,\alpha},\]
where $X^{\vee, \, \alpha}$ denotes the projective dual of $X$ with respect to the embedding given by $\alpha$.
\end{prop}

\begin{proof}
Since $\alpha$ is very ample and $A$ corresponds to the monomial basis of $S_\alpha$, the toric variety $X_A \sset \PP(k^A)^\vee$ is the toric variety $X$ with the embedding defined by $\alpha$. Thus $X^{\vee, \,\alpha} = X_A^\vee$. 

It remains to see that $X_A^\vee = \nabla_A$. To see this, we consider the map $\widetilde{\Phi}_A : (k^*)^{n+1} \to k^A$ as local parameters on the torus in the cone $Y_A \sset (k^A)^\vee$ over $X_A$. Note that since $\alpha$ is very ample, $Y_A$ is an affine toric variety and that $Y_A - \{0\} \to X_A$ is a toric morphism. We claim that $X_A^\vee = \nabla_A$.

For some $f \in (k^A)^\vee$, if $\T_yY_A \sset V(f) \sset (k^A)$ for some $y \in T_{Y_A}$, where $T_{Y_A}$ is the torus in $Y_A$, then $f \in Y_A^\vee$. However, if it holds that $\T_yY_A \sset V(f) \sset (k^A)$ for some torus point $y \in T_Y$, then $f \in \nabla_A^\circ$ (see Definition \ref{A_discrim_def}), since the torus $T_{Y_A} \sset Y_A$ is contained in the smooth locus of $_A$.

Thus we have identified $\nabla_A^\circ$ with a non-empty open (and hence dense) subset of $X^{\vee}_A$ given by hyperplanes containing the tangent space to points on the torus. Note that $\nabla_A^\circ \sset \nabla_A$ is a dense subset by definition. Since both $\nabla_A$ and $X_A^\vee$ are irreducible hypersurfaces in $\PP(k^A)$, we must have $X_A^\vee = \nabla_A$.
\end{proof}

\begin{rmk}\label{rmk_geom_meaning_qs_discr}
This proposition has a very nice geometric meaning. It tells us that for projective toric varieties, the locus of non-quasismooth hypersurfaces in a given complete linear system associated to a very ample class contains (as an irreducible component) the dual to the variety, where the dual is taken with respect to the embedding defined by the very ample line bundle. That is,
\[X_\Sig^{\vee,\alpha} = \nabla_{A_{\Sig,\alpha}}  \sset \YNQS_\alpha \sset \cY_\alpha = \PP(S\alpha) = \PP(k^{A_{\Sig,\alpha}}),\]
where the first containment is as an irreducible component and the second containment is closed.
\end{rmk}

\begin{ex}\label{ex_112_A}
Let $X = \PP(1,1,2)$ and $\alpha = 4 \in \Cl(X) \simeq \ZZ$. Then $S = k[x,y,z]$ where $\deg x = \deg y =1$ and $\deg z = 2$. The monomial basis of $S_\alpha = k[x,y,z]_4$ gives the following matrix
\[A = \begin{pmatrix}
4&3&2&1&0&2&1&0&0\\
0&1&2&3&4&0&1&2&0\\
0&0&0&0&0&1&1&1&2
\end{pmatrix}. \]
For example, the first and second column corresponds to the monomials $x^4$ and $x^3y$ respectively. An element $F \in k^A$ is given by $F(x,y,z) = a_0 x^4 + a_1 x^3y + \cdots + a_4 y^4 + a_5 x^2 z + a_6 xyz +a_7 y^2z + a_8 z^2$ where $a_i \in k$. Note that $S_\alpha = k^A$ is a parameter space for degree 4 hypersurfaces in $\PP(1,1,2)$. 

Then $X_A \simeq \PP(1,1,2)$ and $\cO_X(4)$ is very ample. Explicitly,
\[X_A = \overline{\{ [x^4: x^3y:x^2y^2:xy^3:y^3:x^2z:xyz:y^2:z^2] \, \, | \, \, x,y,z \in k^* \}} \sset \PP^8.\]
\end{ex}

\subsection{Invariance of the $A$-discriminant}

In this section we prove that the $A$-discriminant of a toric variety $X$ is a semi-invariant for the action of the automorphism group of $X$ on $\PP(S_\alpha)$. Since the unipotent radical $U$ of $\aut(X)$ admits no characters, it follows that the discriminant is a true $U$-invariant. To prove this, and to put ourselves in a better position to study the moduli spaces we shall construct in Section \ref{stability_of_hyps}, we prove some results on the geometry of the discriminant locus.

Let $X = \Xs$ be a projective toric variety and $\alpha \in \Cl(X)$ an effective class.  By effective class, we mean a class such that the linear system $|\alpha|$ is non-empty. Let us fix some notation: let $G = \aut_{\alpha}(X)$ and let $T \sset X$ be the torus in $X$. By the definition of toric varieties, the action of $T$ on itself extends to an action on $X$. Thus, we have a map $T \into \aut(X)$, which is injective since $T$ acts faithfully on itself. In fact we have a morphism $T \into \aut^0(X) \sset \aut_\alpha(X)$ since $T$ is connected. 

Recall that $ |\alpha| = \PP(S_\alpha)$ and we write $\cY = \cY_\alpha = |\alpha|$. Consider the projection maps
\[
\begin{tikzcd}
 & X \times \cY \arrow[ld,"\pr_1"'] \arrow[rd,"\pr_2"] & \\
X& &\cY.
\end{tikzcd}
\]
We define the closed set 
\[W = \{ (x,[f]) \in X \times \cY \, \, | \, \, f_i(x) = 0 \, \text{ for } \, 0 \leq i \leq r \} \sset X \times \cY,\] 
where $f_i = \frac{\partial f}{\partial x_i}$ and $r = |\Sig(1)| - 1$. We define the restrictions of the projection maps
\[
\begin{tikzcd}
& W \arrow[ld,"p_1"'] \arrow[rd,"p_2"]&  \\
X& &\cY.
\end{tikzcd}
\]
Clearly we have a $G$-action on $X \times \cY$ given by $g \cdot (x,[f]) = (g \cdot x, g \cdot [f])$. With respect to this action $W$ is an invariant subscheme. Indeed, for $g \in G$ and $(x,[f]) \in W$ we have $\frac{\partial (g \cdot f)}{\partial x_i} (g \cdot x) = f_i(x) = 0$ for every $i$, and hence $g \cdot (x, [f]) = (g \cdot x, g \cdot [f]) \in W$.
We prove the following result describing the flattening stratification of the morphism $p_1$.

\bp\label{p1_fibre}
For any point $x_0 \in X$, the fibre $p_1\inv(x_0) \sset \cY$ is a linear subspace. Moreover, suppose that $x,y \in X$ are in the same $G$-orbit, then $p_1\inv(x) \cong p_1\inv(y)$.
\ep 

\begin{proof}

We can describe the fibre explicitly:
\begin{align*}
p_1\inv(x_0) & = \big\{ (x_0,[f])  \, \, | \, \, \frac{\partial f}{\partial x_i}(x_0) = 0 \, \, \, 0 \leq i \leq r \big\} \\
& = \bigcap_{i=0}^r \big\{ (x_0 , [f]) \, \, | \, \, \frac{\partial f}{\partial x_i}(x_0) = 0 \big\}.\\
\end{align*}
Each of the sets $\{ (x_0 , [f]) \, \, | \, \, \frac{\partial f}{\partial x_i}(x_0) = 0\}$ is the vanishing of a linear polynomial in the coefficients of the polynomial $f$. It follows that $p_1\inv(x_0)$ is the intersection of hyperplanes and thus a linear subspace.

For $g \in G$, we have that 
\[g \cdot p_1\inv(x_0) = p_1\inv(g\cdot x_0),\]
as $g \cdot (x_0,[f]) = (g \cdot x_0 , g\cdot [f])$. In particular, they are all linear subspaces of the same dimension.
\end{proof}

Proposition \ref{p1_fibre} implies that the map $p_1$ is flat when restricted to the the $G$-sweep of the torus.

\bc\label{p1_flat_Y'}
Define $W' = p_1\inv(G \cdot T) \sset W$. Then the map $p_1 |_{W'} : W' \to G \cdot T$ is flat.
\ec

\begin{proof}
Since $B: = G \cdot T = \bigcup_{g \in G} g \cdot T \sset X$ is an open subset of $X$, it is an integral noetherian scheme. Then 
\[B\times \cY = \cY_B \sset \cY_X = X \times \cY\] 
is open and $W\sset \cY_X$ is closed, so  $W' \sset \cY_B$ is a closed subscheme. To see that all fibres over points in $B$ have the same Hilbert polynomial, we observe that, since the torus acts transitively on itself, $G \cdot x = G \cdot T = B$ for all $x \in T$. So applying Lemma \ref{p1_fibre}, we have that all the fibres over $B$ are linear subspaces of the same dimension and thus have the same Hilbert polynomial. Hence we can apply \cite[Theorem III.9.9]{H_AG} and conclude that $p_1|_{W'}$ is a flat morphism.
\end{proof}

By \cite[IV.2, Corollaire 2.3.5 (iii)]{EGA4}, we know that a flat map to an irreducible variety with irreducible generic fibre has an irreducible source. The result holds more generally for open maps (see \cite[Tag 004Z]{stacks}).

\bp 
Let $W'$ be defined as in Corollary \ref{p1_flat_Y'}. Then $W'$ is irreducible.
\ep 

\begin{proof}
Consider the map $p_1|_{W'}:W' \to G \cdot T$. Since $X$ is irreducible and $G \cdot T$ is open, $G \cdot T$ is also irreducible. By Corollary \ref{p1_flat_Y'}, $p_1|_{W'}$ is flat and hence open. By Lemma \ref{p1_fibre}, every fibre is isomorphic to the same projective space, and hence all fibres are irreducible.

Hence we can apply \cite[IV.2, Corollaire 2.3.5 (iii)]{EGA4} to $p_1|_{W'}$ and conclude that $W'$ is irreducible.
\end{proof}

We are now in position to prove the main theorem of this section.

\begin{theorem}\label{invar_A_discrim_locus}
Suppose that $X = \Xs$ is a complete toric variety and that $\alpha\in \Cl(X)$ is a class such that $|\alpha|$ is non-empty. Let $G=\aut_\alpha(X)$ be the automorphism group preserving $\alpha$. Let $A=A_{\Sig,\alpha} \in \ZZ^{r \times N}$ be defined as in Definition \ref{toric_A_disrim}. Then the discriminant locus $\nabla_A \sset \cY$ has the following description.
\[\nabla_A = \overline{ \{ [f] \in \cY \, \, | \, \, \exists \, x_0 \in G \cdot T \text{ such that } f_i(x_0) = 0 \, \, \text{ for all } i\} }.\] In particular, $\nabla_A$ is a $G$-invariant subvariety of $\cY$.
\end{theorem}

\begin{proof}
Note that by the definition of $\nabla$
\[\nabla_A = \overline{ \{ [f] \in \cY \, \, | \, \, \exists x \in T \text{ such that } f_i(x) = 0 \text{ for all } i\} } = \overline{p_2(p_1\inv(T))},\]
where $T \sset X$ is the torus. Then since $ T \sset G \cdot T$, it holds that $p_1\inv(T) \sset  p_1\inv(G \cdot T)$. Thus
\[\overline{p_2(p_1\inv(T))} \sset \overline{p_2(p_1\inv(G \cdot T))}.\]
Applying the definition of $W'$ and the observation that $\overline{p_2(p_1\inv(T))} =\nabla_A$, we conclude
\[\nabla_A \sset \overline{p_2(W')}.\]

Then since $W'$ is irreducible, $\overline{p_2(W')}$ is irreducible. Also note that $\codim p_2(W') \geq 1$, since the quasismooth locus in $\cY$ is open and $W'$ is disjoint from $\YQS$. Hence $\overline{p_2(W')}$ is an irreducible closed subvariety of codimension 1. Then as $\nabla_A$ is an irreducible subvariety of codimension 1, we conclude that 
\[\nabla_A = \overline{p_2(W')},\]
which completes the first part of the theorem.

Now we prove that $\nabla_A$ is $G$-invariant. Note that both maps $p_1$ and $p_2$ are $G$-equivariant since they are restrictions of projections. Then as $G \cdot T$ is $G$-invariant it follows that $W' = p_1\inv (G \cdot T)$ is $G$-invariant and thus $\nabla_A = \overline{p_2(W')}$ is also $G$-invariant.
\end{proof}

\br\label{qs_subset_non_discrim}
This means that the $A$-discriminant will check for hypersurfaces with singularities on the $G$-sweep of the torus in $X$. Note that by Remark \ref{rmk_geom_meaning_qs_discr} we have the inclusion
\[\YQS \subseteq (\cY)_{\Delta_A}. \]
In general these subvarieties do not coincide.
\er

\bc\label{invar_A_discrim}
Keep the notation of Theorem \ref{invar_A_discrim_locus}.  The $A$-discriminant $\Delta_{A}$ is a semi-invariant section for the $G$-action on $\cY_\alpha$ and a true $U$-invariant, where $U \sset G$ is the unipotent radical of $G$.
\ec

\begin{proof}
By definition, $\nabla_A = \V(\Delta_A) \sset \cY$. The automorphism group $G$ acts on  $H^0(\cY,\cO_{\cY}(\deg \Delta_A))$. Since $\nabla_A$ is $G$-invariant, for every $g \in G$ we have that $\V(g \cdot \Delta_A) = \V(\Delta_A)$. Thus $g \cdot \Delta_A = \chi(g) \Delta_A$ for some $\chi(g) \in k^*$. It follows from the group action laws that $\chi(g' g) = \chi(g')\chi(g)$ and thus $\chi : g \mapsto \chi(g)$ is a character. This proves the result.
\end{proof}

\br
The character for which the $A$-discriminant is a semi-invariant is denoted by $\chi_A$ and we denote the degree by $r=\deg \Delta_A$. Thus if we consider the action of $G$ on $\cY_\alpha$ linearised with respect to $(\cO(r),\chi_A)$ (or any linearisation on the ray in $\Pic^G(\cY_\alpha)_\QQ$ defined by this linearisation), it follows that $\cY_\alpha^{\QS}$ lies in the \em naively semistable locus \em defined in \cite{monster}. If this linearisation is also well-adapted (or if $G$ is reductive), then it follows that $\cY_\alpha^{\QS}$ lies in the semistable locus and thus there exists a categorical quotient. See Section \ref{sec_prod_porj_spaces} for an example.
\er

\bex
Let $X= \PP^n$ be standard projective space and $d>0$ a positive integer. In this case $ G \cdot T = \PP^n$, since the action of $G = \GL_{n+1}$ on $\PP^n$ is transitive.  In this case we have that $\nabla_A = \nabla$ is the classical discriminant and that quasismoothness is equivalent to smoothness since $X$ is smooth. Thus 
\[\YQS_d = \PP(k[x_0,\dots,x_n]_d)^{\SM} = \PP(k[x_0,\dots,x_n]_d) - \nabla.\]

This is the ideal situation. The quasismooth locus is given by the vanishing of one invariant section. In general this won't be true. However, we can generalise a little:
for an arbitrary complete toric variety $X$, we have that $G \cdot T = X$ if and only if the action of $G$ on $X$ is transitive, and thus by \cite{bazhov} $X$ is a product of projective spaces.
\eex

\bex\label{ex_rat_scrol_nqs_locus}
Let $X = \PP(1,\dots,1,r) = \proj k[x_0, \dots , x_{n-1},y]$ be the rational cone of dimension $n$, let $G$ the automorphism group of $X$ and let $d = d'r>1$ an integer divisible by $r$. Then $X$ has a single isolated singularity at $(0:\cdots :0:1)$. Let $S_d = k[x_0,\dots,x_{n-1},y]_d$, where $\deg x_i = 1$ and $\deg y = r$. Suppose that $F \in S_d$ is a weighted homogeneous polynomial; then
\[F(x_0,\dots,x_{n-1},y) = \sum_{j=0}^{d'}F_j(x_0,\dots,x_{n-1})y^j,\]
where the $ F_j \in k[x_0,\dots,x_{n-1}]_{d-rj}$ are homogeneous (possibly 0) polynomials of degree $d - jr$. Note that $F_{d'} \in k$ is a constant, write $F_{d'} = c \in k$, then $F(0,\dots,0,1) = c$. Thus $(0: \cdots :0:1)\in \V(F)$ if and only if $c=0$. Moreover, if $c=0$ then $(0: \cdots :0:1)$ is a singular point of $\V(F)$. Indeed, the derivatives are given by
\begin{align*}
\frac{\partial F}{\partial x_i}(x_0,\dots,x_{n-1},y) & = \sum_{j=0}^{d'} \frac{\partial F_j}{\partial x_i}(x_0,\dots,x_{n-1})y^j \\
\frac{\partial F}{\partial y}(x_0,\dots,x_{n-1},y) & = \sum_{j=1}^{d'}jF_j(x_0,\dots,x_{n-1})y^{j-1}.
\end{align*}
Since $d>1$, the $\frac{\partial F_j}{\partial x_i}$ are either $0$ or non-constant homogeneous polynomials in the $x_i$. Thus \linebreak $\frac{\partial F}{\partial x_i}(0,\dots,0,1) = 0$ for every $i$ and $\frac{\partial F}{\partial y}(0,\dots,0,1) = d'c$. Thus the point $(0: \cdots :0:1)$ is a singular point if and only if $c=0$. Note that this means for hypersurfaces in $X$, quasismooth is equivalent to being smooth.

We can write down explicitly the non-quasismooth locus:
\[\YNQS = \cY - \YQS = \nabla_A \cup \V(c) = \V(\Delta_A \cdot c),\]
where $\nabla_A = \PP(1, \dots , 1,r)^{\vee , \, d}$ and we are considering $c$ as a coordinate on $\cY$. In this example $G \cdot T = X - \{(0:\cdots:0:1)\}$. To see this, note that 
\[G = \aut(X) =( (\GG_m \times \GL_{n} ) \ltimes \GG_a^M) / \GG_m,\]
and that $\GL_{n} \into \aut(X)$ acts transitively on the set $\{(x_0: \cdots :x_{n-1}:1) \, \, | \, \, x_i \neq 0 \text{ for some } i\} \sset X$. We prove in Proposition \ref{ss_s_condition1} that the unipotent radical is abelian and that \[M = \begin{pmatrix}
n-1+r \\ r
\end{pmatrix}.\]
\eex

%%%%%%%%%%%%%%%%%%%%%%%%%%%%%%%%%%%%%%%%%%%%%%%%%%%%%%
% MODULI OF QS HYP
%%%%%%%%%%%%%%%%%%%%%%%%%%%%%%%%%%%%%%%%%%%%%%%%%%%%%%
\section{Moduli of quasismooth hypersurfaces}\label{stability_of_hyps}
In this section, we construct coarse moduli spaces of quasismooth hypersurfaces of fixed degree in certain  of toric orbifolds. We prove that quasismooth hypersurfaces of weighted projective space (excluding some low degrees) are stable when the $(\fC)$ condition is satisfied for the action of a grading of the unipotent radical of the automorphism group of this weighted projective space. Once stability is established, we apply the non-reductive GIT Theorem (Theorem \ref{non_red_git_thm}) to conclude that a coarse of moduli space of quasismooth hypersurfaces exists as a quasi-projective variety. Moreover, Theorem \ref{non_red_git_thm} provides a compactification of this moduli space. We also discuss the $(\fC)$ condition and show that it holds for certain weighted projective spaces. We give examples when it does not hold; in this case, one should be able to construct moduli spaces of quasismooth hypersurfaces using the blow-up procedure in \cite{DBHK_proj_compl}.

We also consider smooth hypersurfaces in products of projective spaces and prove that smoothness implies semistability. If we suppose further that the degree is such that the hypersurfaces are of general type, then we prove that smoothness implies stability. Hence we construct a coarse moduli space of such hypersurfaces.
\subsection{$\hat{U}$-stability for quasismooth hypersurfaces in weighted projective space}
Let \linebreak $X = \PP(a_0, \dots , a_n)$ be a well-formed weighted projective space and assume that  $a_0 \leq \cdots \leq a_n$. Let us new give coordinates on $X$ as follows. Let
\[X = \proj k[x_1, \dots , x_{n'} , y_1 , \dots , y_{n_l}],\]
such that $n' +n_l = n+1$ and $\deg x_i < a_n$ and $\deg y_j = a_n$. Note that the $y_j$ are variables with maximum weight. If all the weights coincide then $X = \PP^n$, so we disregard this case.

Let $d$ be a positive integer such that $\lcm(a_j) $ divides $ d$ so that hypersurfaces of degree $d$ of $X$ are Cartier divisors; recall that we call such an integer a \em Cartier degree. \em 

\begin{notation}\label{notation_coord_wps}
We give the parameter space 
\[\cY_d = \Div_X^d = \PP(k[x_1, \dots , y_{n_l}]_d)\]
the following coordinates of the coefficients of the monomials: $(u_0 : \cdots : u_{M'} : v_0 : \cdots : v_M) \in \cY_d$, where the $v_j$ correspond to monomials in the $y_j$ and all have the same total degree, and the $u_i$ are the coefficients of monomials containing an $x_i$ for some $1 \leq i \leq n'$. The integer $M$ is defined by
\[M = \begin{pmatrix}
n_l + d' \\
d'
\end{pmatrix},
\]
where $d' = \frac{d}{a_n}$ and $M'$ is computed in terms of the $a_i's$, but its exact value is not required for the subsequent discussion.
\end{notation}
Recall $G = \autgr(S)$ and that
\[G \simeq \prod_{i=1}^l\GL_{n_i} \ltimes U,\]
where $U$ is the unipotent radical and that $b_1 < \cdots < b_l$ are the distinct values of $a_0 , \dots , a_n$ with each $b_i$ occurring with multiplicity $n_i$. By Proposition \ref{prop_alternative_grading_wps}, the 1-parameter subgroup of $G$ given by 
\[\lambda_{g,N} : t \longmapsto \big( (t^{-N}\id_{n_i})_{i=1}^{l-1},t\id_{n_l} \, , 0 \big),\]
for $N>0$ defines a positive grading of the unipotent radical of $G$ and we define the graded unipotent group $\hat{U}_N = \lambda_{g,N}(\GG_m) \ltimes U$. 
\br
Note that $\hat{U}_N$ depends on the integer $N>0$: that is, for different values of $N$, the subgroups $\lambda_{g,N}(\GG_m) \ltimes U$ are different. However, we shall see in Remark \ref{rmk_Ymin0_locus} that the semistable and stable locus for $\hat{U}_N$ is the same for all $N>>0$.
\er
Let $G$ act on $\cY_d$ with respect to the linearisation $\cO(1)$. Suppose that $\underline{x}^I$ is a monomial in $k[x_1, \dots , x_{n'},y_1,\dots , y_{n_l}]_d$. Then 
\[\lambda_{g,N}(t) \cdot \underline{x}^I = t^{r(N,I)} \underline{x}^I,\]
where $r(N,I) \in \ZZ$ is an integer depending on $N$ and the monomial. Note that for $\underline{y}^I \in k[y_1, \dots , y_{n_l}]_d$ we have that $r(N,I) = -d' = -\frac{d}{a_n}$ is independent of $N$ and
\[\lambda_{g,N}(t) \cdot \underline{y}^I = t^{-d'} \underline{y}^I.\]

Recall the definition of $\Zm$ and $(\cY_d)_{\min}^0$ from Definition \ref{Zmin_Xmin0_def}. Both subsets are defined with respect to a $\hat{U}_N$-action. 
\bl\label{Z_min_locus}
Let $X = \PP(a_0, \dots , a_n)$ and $d \in \Pic(X)$ be a Cartier degree. Fix $\hat{U}_N = \lambda_{g,N}(\GG_m) \ltimes U$. Then 
\[ Z_{\min} = \{(0: \cdots : 0 : v_0 : \cdots : v_M) \, \, | \, \, \exists j: \, \, v_j \neq 0 \} \sset \cY_d\]
and the minimum weight of the $\lambda_{g,N}(\GG_m)$-action on $V = H^0(\cY_d,\cO(1))^\vee$ is $\omega_{\min} = -\frac{d}{a_n}$. 
\el
Note that both $\Zm$ and $\omega_{\min}$ are independent of $N$.
\begin{proof}
As noted above, $\lambda_{g,N}$ acts on monomials containing only variables $y_i$ with weight $-d'=-\frac{d}{a_n}$. Suppose that $\underline{x}^I \in V = k[x_1, \dots , y_{n_l}]_d$ is another monomial containing at least one $x_i$ variable. Then $\lambda_{g,N}(t) \cdot \underline{x}^I = t^{r(N,I)}\underline{x}^I$ and since $\lambda_{g,N}(t) \cdot x_i = t^N x_i$ we have that $r(N,I) > -d'$. Hence $V_{\min} = k[y_1, \dots , y_{n_l}]_d$ and thus 
\[Z_{\min} = \PP(V_{\min})=\{(0: \cdots : 0 : v_0 : \cdots : v_M) \, \, | \, \, \exists j: \, \, v_j \neq 0 \},\]
using Notation \ref{notation_coord_wps}.
\end{proof}

\begin{rmk}\label{rmk_Ymin0_locus}
It follows from the lemma that 
\[(\cY_d)_{\min}^0 = \{ (u_0 : \cdots : u_{M'} :v_0 : \cdots : v_M) \, \, | \, \, \exists j: \, \, v_j \neq 0 \} \sset \cY_d. \]
\end{rmk}

The following lemma follows immediately from Lemma \ref{Z_min_locus}.
\bl\label{Zmin_point}
Let $X = \PP(a_0, \dots , a_n)$ such that $a_n > a_{n-1} \geq \cdots \geq a_0$. Then for every Cartier degree $d \in \ZZ$, we have
\[Z_{\min} = \{(0: \cdots : 0 : 1)\} \sset \cY_d\]
is a point.
\el

\bex\label{ex_wpp112_deg2_weights}
Let $X = \PP(1,1,2) = \proj k[x_1,x_2,y]$ and $d = 2$. Then the monomial basis for $V = H^0(\cY,\cO_{\cY}(1))^\vee$ is $(x_1^2, x_1x_2, x_2^2,y)$. Writing down an arbitrary polynomial
\[ [f(x_1,x_2,y)] =[u_1x_1^2 + u_2x_1x_2 + u_3x_2^2 + vy ] \in \cY_2\]
in coordinates gives $(u_1:u_2:u_3:v) \in \cY_2$.  Now let us calculate the weights for the grading $\GG_m$-action defined by $\lambda_{g,N}$, with $N>0$. For positive integers $i$ and $j$ such that $i+j =2$ we have
\[\lambda_{g,N}(t) \cdot x_1^ix_2^j = \left( t^Nx_1\right)^i\left(t^Nx_2 \right)^j= t^{2N}x_1^ix_2^j \, \,\text{ and } \, \, \lambda_{g,N}(t) \cdot y = t\inv y.\]
Hence we have two distinct weights 2 and 1 and the decomposition into weight spaces is given by
\[V = V_{2N} \oplus V_{-1}=\Span(x_1^2, x_1x_2, x_2^2) \oplus \Span (y).\]
Thus $Z_{\min} = \PP(V_{-1}) = \{(0:0:0:1)\}$.
\eex
\begin{notation}
For the rest of this section we fix $N>0$ and $\hat{U} = \lambda_{g,N} \ltimes U$.
\end{notation}
\begin{prop}\label{Uhat_stab}
Let $X = \PP(a_0, \dots , a_n)$ be a well-formed weighted projective space and $d$ be a Cartier degree. Denote $\cY = \cY_d$. Then we have the following inclusion:
\[\YQS \sset \cY_{\min}^0 - U \cdot Z_{\min}.\]
\end{prop}

\begin{proof}
We begin by observing that 
\[\cY - \cY_{\min}^0 = \{(u_0 : \cdots : u_{M'} : 0 : \cdots : 0)\}.\]
Take some $f \in \cY - \cY_{\min}^0$. We know that $f$ contains no monomials made up of only the $y_i$. Thus $(0 : \dots : 0 : 1) \in X$ will be a common zero for all $\frac{\partial f}{\partial x_i}$ and $\frac{\partial f}{\partial y_i}$ since $d$ is a Cartier degree (as monomials of the form $y_{n_l}^{d'-1}x_i$ can never be homogeneous of degree $d$). It follows that $f$ is not quasismooth and hence
\[\YQS \sset \cY_{\min}^0.\]
Suppose that $f \in Z_{\min}$. Then $f$ is a polynomial in the $y_i$ and so $(1:0: \cdots : 0) \in X$ will be a common zero for all $\frac{\partial f}{\partial x_i}$ and $\frac{\partial f}{\partial y_i}$. Thus $f$ is not quasismooth and we have that
\[Z_{\min} \sset \YNQS.\]
Since $\YNQS$ is a $G$-invariant subset, it follows that
\[U \cdot Z_{\min} \sset \YNQS\]
and so we conclude that
\[\YQS \sset \cY_{\min}^0 - U \cdot Z_{\min}.\]
\end{proof}

We show that the condition $(\fC)$ (see Definition \ref{ss_equals_s}) holds in the case for weighted projective spaces of the form $\PP(a, \dots ,a, b, \dots ,b)$.

\begin{prop}\label{ss_s_condition1}
Suppose $X = \PP(a, \dots , a , b , \dots , b) = \proj k[x_1, \dots x_n,y_1,\dots , y_m]$ is a weighted projective space such that $b>a$ and let $d>b+1$ be a Cartier degree (so that $ab$ divides $d$). 

Then the graded automorphism group of $S = k[x_1, \dots x_n,y_1,\dots , y_m]$ is of the following form
\[\autgr(S) = (\GL_n \times \GL_m) \ltimes \GG_a^L,\]
where $L\geq 0$.
In particular, the unipotent radical $U = \GG_a^L$ is abelian.
Moreover, the action of $G$ on $\cY_d$ with respect to $\oone$ satisfies the condition $(\fC)$; that is, the stabiliser group is trivial
\[\stab_U([f]) = \{e\}\]
for every $[f] \in \Zm^{\ss,R} \sset \cY_d$, where $R \cong \GL_n \times \GL_m$.
\end{prop}

\begin{proof}
Let $G = \autgr(S)$; then a general automorphism in the unipotent radical $\phi \in U \sset G$ is given by
\[
    \phi : \left\{ \begin{array}{lr}
        x_i   \longmapsto   x_i \\
        y_j   \longmapsto   y_j+ p_{\phi,j}(x_1, \dots , x_{n})
        \end{array}\right.
\]
for $1 \leq i \leq n$ and $1 \leq j \leq m$ and with $p_{\phi,j} \in k[x_1, \dots , x_{n}]_r$.
Composing two such elements $\phi,\psi \in U$ gives
\[
    \phi \circ \psi : \left\{ \begin{array}{lr}
        x_i  \longmapsto  x_i \\
        y_j  \longmapsto  y_j+ p_{\phi,j}(x_1, \dots , x_{n})\!\!\!\!&+ \,\, p_{\psi,j}(x_1, \dots , x_{n}).
        \end{array}\right. 
\]
It follows that any two automorphisms commute and hence $U$ is abelian and thus
\[U \simeq \GG_a^L.\]

Let us prove the second statement. 
Let $d' = \frac{d}{b}$ and recall that $\Zm = \PP(k[y_1, \dots , y_m]_{d'})$.
Take some $[f] \in \Zm^{\ss,R}$ and $\phi \in \stab_U([f])$.
Assume that $\phi$ in non-trivial.

Since the $\GL_n$ acts trivially on $\Zm$ it follows that $\Zm^{\ss,R} = \Zm^{\ss,\GL_m}$.
Note that this only holds for semistability.
Then considering the standard maximal torus in $\GL_m$, the reductive Hilbert-Mumford criteria imply that the barycenter of the section polytope of $\cO_{\Zm}(d')$ is contained in the Newton polytope of $[f]$.
Since $d>b+1$, it follows that each variable $y_j$ appears in $[f]$ and thus $\phi \cdot [f]$ must contain at least one $x_i$.
Hence $\phi \cdot [f] \neq [f]$ and we have a contradiction.
\end{proof}

This encompasses many interesting examples of weighted projective hypersurfaces.
For example smooth degree 4 hypersurfaces in $\PP(1, \dots , 1 ,2)$ from a large family of Fano varieties of Fano index 2.
Moreover, degree 2 smooth del Pezzo surfaces are completely described by degree 4 surfaces in $\PP(1,1,1,2)$ and a blow-up of the corresponding moduli space has been shown to describe the K-stable degree 2 del Pezzo surfaces \cite{OSS}.
Other examples of interest \cite{LP20} are the K-stable del Pezzo surfaces of Fano index 2 which present themselves as degree $2a$ hypersurfaces in $\PP(1,1,a,a)$ for $a\geq 2$.

\br 
The condition $(\fC)$ is not satisfied for every weighted projective space; for example, consider $X= \PP(1,2,3) = \proj k[x,y,z]$ and $d = 6$. Then 
\[Z_{\min} = \{(0: \cdots : 0 :1)\} = \left\{ \left[ z^2\right]\right\}\]
is a point by Lemma \ref{Zmin_point} and corresponds to the hypersurface defined by $z^2$. However, the additive 1-parameter subgroup of $U$
\[a(u) : y \longmapsto y + ux^2\]
acts trivially on $Z_{\min}$.

There are other examples of weighted projective space for which the condition $(\fC)$ is satisfied. For example, let $X = \PP(2,2,3,3,5)$ with coordinates $x_1,x_2,y_1,y_2,z$ such that $\deg x_i = 2$, $\deg y_i = 3$ and $\deg z = 5$. Let $d = 20$ and note that
\[\aut(X) = ((\GL_2 \times \GL_2 \times \GG_m) \ltimes (\GG_a)^4) / \lama(\GG_m).\]
Again we have that $Z_{\min}$ is a point corresponding to the hypersurface $z^4$. Then the action of $(\GG_a)^4$ is trivial on coordinates $x_1,x_2,y_1,y_2$ and on $z$ the action is defined by 
\[(A_1,A_2,A_3,A_4) \cdot z = z+ A_1x_1y_1 + A_2x_1y_2 + A_3x_2y_1 + A_4x_2y_2,\]
where $(A_1,A_2,A_3,A_4) \in (\GG_a)^4$. It follows that the $(\GG_a)^4$-stabiliser of $\left[z^4\right]$ is trivial.
\er

\br\label{rmk_ss_s_intrinsic}
Let $X = \PP(a_0 ,\dots , a_n)$ be a weighted projective space and $d = \lcm(a_0, \dots ,a_n)$. If the condition $(\fC)$ holds for the action of $\hat{U} \sset G$ on $\cY_d$ with respect to $\oone$, an induction argument on $l$ shows it also holds for the $G$-action on $\cY_{ld}$ for every $l >0$.
Hence the condition $(\fC)$ for $\hat{U}$ is an intrinsic property of $X$. Thus we may talk about $X$ satisfying the condition $(\fC)$.
\er

\br\label{rmk_ep_lin}
Let us revisit the definition of a well-adapted linearisation from Remark \ref{rmk_defn_well_adpated}. The well-adapted requirement is needed in the proof of the $\hat{U}$-theorem \cites{BDHK_gr_uni,DBHK_proj_compl} and insures that semistability and stability coincide for the grading $\GG_m$ once a high tensor power of the linearisation is taken. However, this high power is determined only by the $U$-invariants. Hence if we take $N>d'$ then we have the weight diagram shown in Figure \ref{fig_wt_distribution}, where $\omega_{\min +1}$ is the next biggest weight and $r(N)$ grows linearly with $N$, see Example \ref{ex_wpp112_deg2_weights}.\footnote{This choice of lower bound $N > d'$ is not optimal; we could take a smaller $N$.} Thus we can choose some $N>>0$ such that the linearisation $\oone$ is well-adapted and we can readily apply the $\hat{U}$-theorem.
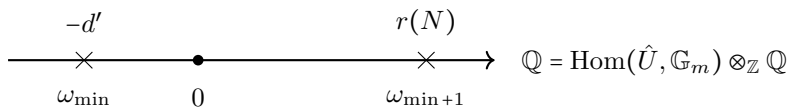
\begin{figure}[h]
\centering
\begin{tikzpicture}

\draw[thick, ->] (-2.5,0) -- (3.9,0);
\node at (0,0) [circle,fill=black, inner sep=0pt,minimum size=4pt] {};

\node at (0,-0.5) {$0$};
\node at (-1.5,-0.5) {$\omega_{\min}$};
\node at (-1.5,0) [cross=3.5] {};
\node at (-1.5,0.5) {$-d'$};
\node at (3,-0.5) {$\omega_{\min+1}$};
\node at (3,0.5) {$r(N)$};
\node at (3,0) [cross=3.5] {};
\node at (6,0) {$\QQ = \Hom(\hat{U},\GG_m) \otimes_\ZZ \QQ$};

\end{tikzpicture}
\caption{The weight diagram for $\hat{U} = \lambda_{g,N}(\GG_m) \ltimes U$ for $N>>0$.}\label{fig_wt_distribution}
\end{figure}
For the details of the proof we refer to \cites{BDHK_gr_uni,DBHK_proj_compl} and for a discussion on the variation of the grading 1-parameter subgroup we refer to \cite{variationNRGIT}. From now on we shall assume that we have taken $N >>0$.
\er
\bc\label{cor_Uhat_stab}
Let $X = \PP(a_0, \dots , a_n)$ be a weighted projective space and $d$ be a Cartier degree. Assume that $X$ satisfies the condition $(\fC)$ for the action of $\hat{U}$ on $\cY_d$ with respect to $\oone$, where $\hat{U} = \lambda_{g,N}(\GG_m) \ltimes U \sset G = \autgr(S)$ for $N>>0$. Then the following statements hold.
\begin{enumerate}
\item The quotient morphism
\[q_U :\cY_{\min}^0 \to \cY_{\min}^0 / U\]
is a principal $U$-bundle. 

\item The quotient morphism
\[q_{\hat{U}} : \cY_d^{ \s , \hat{U}} \longrightarrow \cY_d \sslash \hat{U} \]
is a projective geometric quotient, where $\cY_d^{ \s , \hat{U}} = \cY_{\min}^0 - U \cdot \Zm$.
\item The subset
\[ \YQS /  \hat{U} = q_{\hat{U}} (\YQS) \subseteq \cY_d \sslash \hat{U}\]
is open, and thus $\YQS / \hat{U}$ is quasi-projective.
\end{enumerate}
\ec
\begin{proof}
The first statement follows directly from the $\hat{U}$-Theorem (Theorem \ref{thm_Uhat_thm}) and the second statement follows from the fact that a geometric quotient is an open map (since it is a topological quotient) and that Proposition \ref{Uhat_stab} we that $\YQS \sset \cY^{s,\hat{U}}$ is an open subset.
\end{proof}

\subsection{Stability of quasismooth hypersurfaces in weighted projective space}

We present a proof that quasismooth hypersurfaces are stable using the non-reductive Hilbert-Mumford criteria of Theorem \ref{thm_nrgit_hm_2}. The Hilbert-Mumford criterion says that if $G$ is a linear algebraic group with graded unipotent radical acting on a projective variety $Y$, then a point $y \in Y$ is stable if and only if every $G$ translate $g \cdot y$ is stable for a maximal torus $T\sset G$ containing the grading $\GG_m$. We shall prove stability of quasismooth hypersurfaces for a maximal torus $T$ and then use the fact that the quasismooth locus is invariant under the action of the automorphism group and the NRGIT Hilbert-Mumford criterion to deduce stability for $G$. The proof of $T$-stability uses the Newton polytope of a hypersurface, which we define as the weight polytope for the canonical maximal torus.

Let $X = \PP(a_0, \dots , a_n)$ be a well-formed weighted projective space such that $a_{i}\leq a_{i+1}$ and $d$ be a Cartier degree. Suppose that $T \sset G = \autgr(S)$ is the maximal torus of $G$ given by diagonal matrices and define $\overline{T} = T / \lama(\GG_m)$ to be the quotient by the 1-parameter subgroup $\lama$. Recall from Section \ref{sec_s_global_stab} that the stability of a hypersurface with respect to $T$ is determined by its weight polytope considered inside the character space $H = X^*\left(\overline{T}\right) \otimes_\ZZ \QQ \simeq \QQ^{n+1} / \QQ \cdot \underline{a}$. The weight polytope is a subpolytope of the section polytope.

\bd\label{def_sec_polytope}
Let $X = \PP(a_0, \dots , a_n)$ be a well-formed weighted projective space and $d \in \ZZ_{>0}$ be a Cartier degree. Consider $H^0(X,\cO_X(d)) = k[x_0, \dots , x_n]_d$ and define 
\[A = \{(i_0, \dots , i_n) \in \ZZ^{n+1} \, \, | \, \, i_j \geq 0 \text{ and }  \sum a_ji_j = d\}\]
to be the set of exponent vectors of the monomials. The set $A$ is precisely the set of torus weights and let $\tilde{P_d} = \conv(A) \sset \QQ^{n+1}$. We define the \em section polytope \em to be $P_d \sset H=\QQ^{n+1}/ \QQ \cdot \underline{a}$, the image of $\tilde{P_d}$ in $H$. Explicitly, choosing a basis for $H$, we have that 
\[P_d = \conv\left(\frac{d}{a_0}e_0 , \dots , \frac{d}{a_{n-1}}e_{n-1} , -\frac{d}{a_n^2}(a_0e_0 + \cdots + a_{n-1}e_{n-1})\right).\]
Note that $P_d$ contains the origin.
\ed

\bd
For a degree $d$ hypersurface $Y \sset X$, we define the \em Newton polytope of $Y$ \em by
\[\NP(Y) = \conv(\wt_{\overline{T}}(Y)) \sset H = X^*\left(\overline{T}\right) \otimes_\ZZ \QQ.\]
\ed

\bex
Let $X = \PP(1,1,2)$ and consider Figure \ref{fig_112_sec_polytope} where we have the section polytope of $\cO_{X}(4)$ and the Newton polytope of $f = xy^3 + x^2z+y^2z+z^2 $. Note that by $\hat{U}$-stability, for a point to be semistable we must have that the circled vertex corresponding to $\Zm = \{z^2\}$ appear in the weight space and thus also in the Newton polytope.

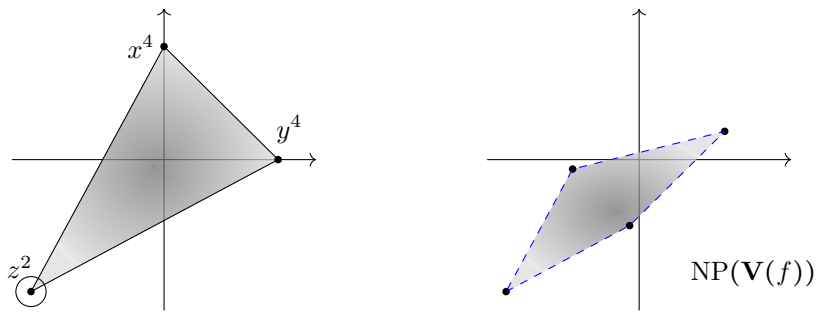
\begin{figure}[h]
\centering
\begin{tikzpicture}

\draw  [->] (0,-2) -- (0,2);
\draw [->] (-2,0) -- (2,0);
\fill[shading=radial, nearly transparent] (-1.75,-1.75) to (1.5,0)  to (0,1.5);

\node [above] at (-1.90,-1.7) {$z^2$};
\node [circle, fill, inner sep=1.0pt] at (-1.75,-1.75) {};

\node [left] at (0,1.5) {$x^4$};
\node [circle, fill, inner sep=1.0pt] at (0,1.5) {};

\node [above] at (1.65,0.1) {$y^4$};
\node [circle, fill, inner sep=1.0pt] at (1.5,0) {};

\node [circle, draw, inner sep=4.0pt] at (-1.75,-1.75) {};

\draw  [-] (-1.75,-1.75)-- (1.5,0);
\draw  [-]  (1.5,0) -- (0,1.5);
\draw [-] (0,1.5) -- (-1.75,-1.75);
\end{tikzpicture}
\hspace{2cm}
\begin{tikzpicture}

\draw  [->] (0,-2) -- (0,2);
\draw [->] (-2,0) -- (2,0);
\fill[shading=radial, nearly transparent] (-1.75,-1.75) to (-0.125,-0.875) to (1.125,0.375)  to (-0.875,-0.125);

\node [circle, fill, inner sep=1.0pt] at (-1.75,-1.75) {};

\node [circle,fill, inner sep=1pt] at  (1.125,0.375) {};
%\node [circle, draw, inner sep=4.0pt] at  (1.125,0.375) {};

\node [circle,fill, inner sep=1pt] at  (-0.125,-0.875) {};
%\node [circle, draw, inner sep=4.0pt] at  (-0.125,-0.875) {};

\node [circle,fill, inner sep=1pt] at  (-0.875,-0.125) {};
%\node [circle, draw, inner sep=4.0pt] at  (-0.875,-0.125) {};

\node [] at (1.5,-1.5) {$\NP(\V(f))$};

%\node [circle, draw, inner sep=4.0pt] at (-1.75,-1.75) {};
\draw[dashed,blue] (-1.75,-1.75) to (-0.125,-0.875) to  (1.125,0.375)  to (-0.875,-0.125) to (-1.75,-1.75);

\end{tikzpicture}
\caption{On the left is the section polytope $P$ of $\cO(4)$. On the right is the Newton polytope of $\V(f)$ where $f = xy^3 + x^2z+y^2z+z^2$.}\label{fig_112_sec_polytope}
\end{figure}
\eex

\bl\label{lemma_qs_NP_contains_origin}
Suppose that $Y \sset X = \PP(a_0, \dots ,a_n)$ is a quasismooth hypersurface of degree \linebreak$d \geq \max\{a_0, \dots a_n\}+2$ defined by a weighted polynomial $f \in k[x_0, \dots , x_n]_d$. Then $\NP(Y) \sset H$ contains the origin.
\el

\begin{proof}
Let $P \sset H$ be the section polytope of $\cO_X(d)$ as in Definition \ref{def_sec_polytope}. Then $P$ contains the origin. Suppose that $\NP(Y)^\circ$ does not contain the origin. Then $\NP(Y)$ is contained in a closed half-space of $H$ defined by a hyperplane passing through the origin. Thus all the monomials of $f$ lie in this half space and since $Y$ is quasismooth this is impossible by Theorem \ref{thm_flecther_qs_criterion}. Indeed, suppose that this is the case, then since $d \geq \max(a_0, \dots , a_n)+2$, a vertex and immediate surrounding lattice points of the section polytope will certainly lie in the complementary open half-space. This is a contradiction by Lemma \ref{lemma_fletcher_qs_crit}.
\end{proof}

We are now in a position to prove the main theorem of the section.
\begin{theorem}\label{thm_main_NP_proof}
Let $X = \PP(a_0 , \dots , a_n)$ be a well-formed weighted projective space such that $n>1$ which satisfies the condition $(\fC)$ and let $d \geq \max\{a_0, \dots , a_n\}+2$ be a Cartier degree. Let $G = \autgr(S)$ be the graded automorphism group of the Cox ring and consider the action of $G$ on $\cY = \cY_d$ linearised with respect to $\oone$. Define the graded unipotent radical $\hat{U} = \lambda_{g,N}(\GG_m) \ltimes U \sset G$ for some fixed $N >> d$. We have the inclusion
\[\YQS \sset \cY^{\s,G}.\]
In particular, there exists a geometric quotient $\YQS / G$ and hence a coarse moduli space of quasismooth hypersurfaces as a quasi-projective variety. Moreover, $\cY \sslash \! G$ is a compactification of $\YQS / G$.
\end{theorem}

\begin{proof}
We drop the linearisation from the notation. Note that it suffices to prove that $\YQS \sset \cY^{\s,T}$, since by the non-reductive Hilbert-Mumford criterion of Theorem \ref{thm_nrgit_hm_2} we get $\YQS \sset \cY^{\s,G}$. Indeed, as $\YQS$ is $G$-invariant, we have that $\YQS \sset g \cdot \cY^{\s,T}$ and hence
\[\YQS \sset \bigcap_{g \in G} g \cdot \cY^{\s,T} = \cY^{\s,G},\]
by Theorem \ref{thm_nrgit_hm_2}.

Let us prove that $\YQS \sset \cY^{\s,T}$. Suppose that $Y \sset X$ is a quasismooth hypersurface of degree $d$. Then by Lemma \ref{lemma_qs_NP_contains_origin}, the polytope $\NP(Y)$ contains the origin $O$. Thus by the Hilbert-Mumford criterion of Theorem \ref{thm_stab_HM_polytope} quasismooth hypersurfaces are $T$-stable for the twisted linearisation $\oone$.
\end{proof}

\br
In the case where the condition $(\fC)$ is not satisfied, there is a blow-up procedure outlined in \cite{DBHK_proj_compl} where one performs a sequence of blow-ups of the locus in $\cY$ where there is a positive dimensional $U$-stabiliser.
Using this procedure it is expected that we can remove the requirement that the $(\fC)$ condition holds.
This will be pursued in further work.
\er

\subsection{Explicit construction for $\PP(1,\dots,1,r)$}

We provide an explicit constructions of a coarse moduli spaces of quasismooth hypersurfaces in the case where $X = \PP(1, \dots , 1 , r) = \proj k[x_1, \dots , x_n , y]$ and $d = d' \cdot r$ with $d' >0$ and $n >1$. This example illustrates the `quotienting in stages' procedure.

We give a direct construction of these coarse moduli space using Lemma \ref{lemma_quotients_in_stages}. 

\begin{theorem}
Let $X =  \PP(1, \dots , 1 , r)$ and $d = d' \cdot r$ be a Cartier degree such that $d \geq r+2$. Let $\cY = \PP(k[x_1, \dots , x_n , y]_d)$ be the parameter space of degree $d$ hypersurfaces. Then there exists a geometric quotient for the $G$-action on $\YQS$
\[\YQS \longrightarrow\YQS / G \]
which is coarse moduli space and a projective over affine variety.
\end{theorem}

\begin{proof}
Let $c \in H^0(\cY,\oone) = (k[x_1, \dots , x_n ,y]_d)^\vee$ be the section corresponding to the coefficient of the monomial $y^{d'}$.  By Example \ref{ex_rat_scrol_nqs_locus}, we have that $\YQS = \cY_{c \cdot\Delta_A}$ and hence $\YQS$ is an affine variety. Note that we have the inclusion $\YQS \sset \cY_c$. In this case, $Z_{\min} = \{(0: \cdots : 0 : c) \, \, | \, \, c \neq 0\}$ is a point, and so by Remark \ref{rmk_triv_U_quot} the quotient 
\[q_U : \cY_{\min}^0 \longrightarrow \cY_{\min}^0/U\]
from Corollary \ref{cor_Uhat_stab} is a trivial $U$-bundle. Hence $\cY_{\min}^0/U$ is affine by \cite[Theorem 3.14]{asok2009a1}. Thus $\YQS \to \YQS /U$ is a trivial bundle and $Q = \YQS/U$ is an affine variety.

Consider the action of $R = G /U $ on $Q$. Since $\YQS$ is affine, Lemma \ref{lemma_quotients_in_stages} implies that $Q$ admits a geometric quotient by $R$ if and only if all the $G$-orbits are closed in $\YQS$. Then as $d \geq r+2$, Theorem \ref{finite_stabilisers_thm} implies that all stabiliser groups are finite giving that the action on $\YQS$ is closed. Hence we have a geometric quotient
\[\YQS / G = Q / R.\]
Since $Q$ is an affine variety and $Q/R$ is a reductive quotient, we conclude that $\YQS / G$ is a projective over affine variety.
\end{proof}

\bex
Suppose that $X = \PP(1,1,2)$ and $d = 6$. Then quasismooth hypersurfaces are exactly Petri special curves of genus 4 in $X$. Thus $(\cY_6)^{\QS} / G$ is an projective over affine coarse moduli space of Petri special curves. This moduli space is a divisor on the moduli space of genus 4 curves (see \cite{tommasi}).
\eex

\bex
Let $X = \PP(1,1,1,2)$ and consider $d = 4$. In this case quasismoothness coincides with smootheness. The smooth hypersurfaces are exactly degree 2 del Pezzo surfaces. Hence $(\cY_4)^{\QS} / G$ is a projective over affine coarse moduli space of degree 2 del Pezzo surfaces.
\eex

\subsection{Stability of hypersurfaces in products of projective space}\label{sec_prod_porj_spaces}

Let $X = \PP^n \times \PP^m$ and $(d,e) \in \ZZ^2 \simeq \Pic(X)$. Then $\autgr(S) = \GL_{n+1} \times \GL_{m+1}$ and we consider the action of $G$ on the projective space $\cY = \PP(k[x_0, \dots , x_n ; y_0, \dots , y_m]_{(d,e)})$. Note that every stabiliser contains the subgroup 
\begin{align*}
\lambda : \GG_m^2& \longrightarrow \autgr(S) \\
 (t_1,t_2) &\longmapsto (t_1 I_{n+1} , t_2 I_{m+1}).
\end{align*}
We may replace the action of $\autgr(S)$ with the action of the subgroup $G = \SL_{n+1} \times \SL_{m+1} \sset \autgr(S)$, as the orbits are the same. By doing this, we remove the global stabiliser and moreover, we are now in the situation where $G$ has no non-trivial characters and so $\Delta_A$ is a true invariant for the $G$-action.

\begin{prop}\label{prop_prod_proj_finite_stabs}
Let $X = \PP^n \times \PP^m$ be a product of projective spaces and $Y \sset X$ be a smooth hypersurface of degree $(d,e) \in \ZZ^2$. If $d > n+1$ and $e > m+1$, then $\dim \stab_G(Y) = 0$.
\end{prop}

\begin{proof}
Since $Y$ is a proper algebraic scheme, by \cite[Lemma 3.4]{matsumura1967representability}, we can identify the Lie algebra of the automorphism group of $Y$ with the vector space $H^0(Y,\cT_Y)$ of global vector fields on $Y$, where $\cT_Y$ is the tangent sheaf. If we show that $ H^0(Y,\cT_Y) = 0$, then we can conclude that $\dim \aut(Y) = 0$ and since $\stab_{G}(Y) \sset \aut(Y)$ we have that $\dim \stab_G(Y) = 0$.

Let us prove that $H^0(Y,\cT_Y) = 0$. Let $N = \dim Y = n+m -1$; then by Serre duality
\[H^0(Y,\cT_Y) \simeq H^N(Y, \Omega_Y \otimes \omega_Y)^\vee,\]
where $\omega_Y$ is the canonical line bundle of $Y$. Let $\cO_Y(1,1)$ be the restriction of $\cO_X(1,1) = \cO_{\PP^n}(1) \boxtimes \cO_{\PP^m}(1)$ to $Y$. By the adjunction formula we have that $\omega_Y \cong \cO_Y(d-n-1,e-m-1)$ and hence by our assumption we have that $\omega_Y$ is very ample. Then by Kodiara-Nakano vanishing (\cite[Theorem 1.3]{esnault1992vanishing}) we have that $H^N(Y,\Omega_Y(d-n-1,e-m-1)) = 0$. Hence
\[H^0(Y,\cT_Y) \simeq H^N(Y,\Omega_Y(d-n-1,e-m-1))^\vee = 0.\]
We conclude that $\dim \aut(Y) = 0$.
\end{proof}

\br
Note that $\aut(Y)$ may not be a linear algebraic group, and in general is only locally linear algebraic.
\er

\begin{theorem}\label{thm_prod_proj_space_moduli}
Let $X = \PP^n \times \PP^m$ be a product of projective spaces and $Y \sset X$ be a smooth hypersurface of degree $(d,e) \in \ZZ^2$. Consider the action of $G = \SL_{n+1} \times \SL_{m+1}$ on  $\cY = \PP(k[x_0, \dots , x_n ; y_0, \dots , y_m]_{(d,e)})$ with linearisation given by $\cO_{\cY}(1)$.  Then we have the open inclusion
\[\cY^{\SM} \sset \cY^{\ss}(\cO(1)),\]
where $\cY^{\SM}$ is the of smooth hypersurfaces. If $d > n+1$ and $e > m+1$ then we have the open inclusion
\[\cY^{\SM} \sset \cY^{\s}(\cO(1)).\]
In particular, there exists a coarse moduli space of smooth hypersurfaces of degree $(d,e)$.
\end{theorem}

\begin{proof}
First, note that the discriminant $\Delta_A$ is a true invariant for the $G$-action since $G$ has no non-trivial characters. Then by Theorem \ref{invar_A_discrim_locus}, we have that
\[\cY^{\SM} = \cY_{\Delta_A},\]
since $G$ acts transitively on $X$ and hence $\cY^{\SM} \sset \cY^{\ss}$. Finally, if $d > n+1$ and $e > m+1$, by Proposition \ref{prop_prod_proj_finite_stabs}, we have that the stabiliser for every point in $\cY^{\SM}$ is finite and hence all the orbits are closed. It follows that $\cY^{\SM} \sset \cY^{\s}$.
\end{proof}

\bibliography{referencesVicky}
\end{document}